\documentclass[10pt]{amsart}
\usepackage[pdftex]{geometry}
\usepackage[final]{graphicx}
\usepackage{amsmath}
\usepackage{amsfonts}
\usepackage{amsthm}
\usepackage{amssymb}

\usepackage{concrete}
\usepackage{euler}

\newtheorem{theorem}{Theorem}
\newtheorem{proposition}[theorem]{Proposition}
\newtheorem{lemma}[theorem]{Lemma}

\newtheorem{definition}[theorem]{Definition}

\numberwithin{theorem}{section}
\newcommand{\ignore}[1]{}

\newcommand{\hlineFix}{\rule{0pt}{17pt}}

\newcommand{\C}{\mathbb{C}}

\newcommand{\N}{\mathbb{N}}

\newcommand{\Q}{\mathbb{Q}}

\newcommand{\Z}{\mathbb{Z}}

\newcommand{\spmox}[1]{\lambda_{#1}}
\newcommand{\spmoh}{\spmox{2}}
\newcommand{\spmot}{\spmox{3}}
\newcommand{\spmoq}{\spmox{4}}

\newcommand{\dpmox}[1]{(p-1)/{#1}}
\newcommand{\dpmoh}{\dpmox{2}}

\newcommand{\fpmox}[1]{\frac{p-1}{#1}}
\newcommand{\fpmoh}{\fpmox{2}}

\newcommand{\fppox}[1]{\frac{p+1}{#1}}
\newcommand{\fppoh}{\fppox{2}}

\newcommand{\ph}[2]{\left ( #1 \right)_{#2}}

\newcommand{\citep}[2]{\cite[p.~#2]{#1}}

\newcommand{\Fbar}{\bar F}
\newcommand{\Gbar}{\bar G}

\newcommand{\Fdeg}{\widetilde{F}}

\newcommand{\elide}[1]{}

\begin{document}

\title{Streamlined WZ method proofs of Van Hamme supercongruences}

\author{Andr\'es Valloud}

\date{October 10, 2025}

\begin{abstract}
Using the WZ method to prove supercongruences critically depends on an inspired WZ pair choice.  This paper demonstrates a procedure for finding WZ pair candidates to prove a given supercongruence.  When suitable WZ pairs are thus obtained, coupling them with the $p$-adic approximation of $\Gamma_p$ by Long and Ramakrishna enables uniform proofs for the Van Hamme supercongruences (B.2), (C.2), (D.2), (E.2), (F.2), (G.2), and (H.2).  This approach also yields the known extensions of (G.2) modulo $p^4$, and of (H.2) modulo $p^3$ when $p$ is $3$ modulo $4$.  Finally, the Van Hamme supercongruence (I.2) is shown to be a special case of the WZ method where Gosper's algorithm itself succeeds.
\end{abstract}


\maketitle

\thispagestyle{empty}

\section{Introduction}

In 1997, Van Hamme \cite{VanHamme97} stated thirteen $p$-adic analogues to the Ramanujan hypergeometric series for $1 / \pi$ of 1914.  These analogues claim that certain truncated hypergeometric series satisfy congruences that hold modulo unexpectedly large prime powers.  When this phenomenon occurs, the resulting identities are called supercongruences.  The analogues of Van Hamme are labeled (A.2) through (M.2).  Several authors have contributed proofs of these using a variety of methods, including Van Hamme ((C.2), (H.2), (I.2)) \cite{VanHamme97}, Kilbourn ((M.2)) \cite{Kilbourn}, McCarthy and Osburn ((A.2)) \cite{McCarthyOsburn}, Mortenson ((B.2)) \cite{Mortenson}, Zudilin ((B.2)) \cite{Zudilin09}, Long ((B.2)) \cite{Long}, Long and Ramakrishna ((D.2), (H.2)) \cite{LongRama}, Swisher ((C.2), (E.2), (F.2), (G.2), (L.2)) \cite{Swisher15}, He ((E.2), (F.2)) \cite{He15}, and Osburn and Zudilin ((K.2)) \cite{OsburnZudilin15}.  The last of these supercongruences was proved in 2016.  For a table summarizing the Van Hamme supercongruences, see Van Hamme \cite{VanHamme97} and Swisher \cite{Swisher15}.

For a family of primes $p$, the typical Van Hamme supercongruence is an identity of the form
\begin{equation} \label{e-van-hamme}
\sum_{n=0}^{\frac{p-1}{d}} u(n)c^n \cdot \frac{\ph{1/a}{n}^m}{\ph{1}{n}^m} \equiv f(p) \pmod{p^r},
\end{equation}

\noindent
where $d \in \N$ and $c$ is some constant.  The positive integers $m, r$ are relatively small, and the rational $a \neq 0$ often has small numerator and denominator.  The expression $u(n)$ is a polynomial in $n$.  The function $f(p)$ is usually a simple monomial depending on $p$ perhaps accompanied by factors such as a power of $-1$ or an instance of Morita's $p$-adic gamma function $\Gamma_p$.  Other supercongruences may have a product of different Pochhammer symbols, each with their own value of $a$.

Although the Van Hamme supercongruence expressions are similar to each other, their proofs can vary substantially.  Often, with these and other supercongruence proofs, there is a unique element of ingenuity one wishes could be developed into a widely applicable method.  This is especially so in view of Sun's 99-page paper \cite{Sun} which contains 100 supercongruence conjectures, as well as the various supercongruence conjectures listed by Zudilin \cite{Zudilin09}, Swisher \cite{Swisher15}, and many others.

The Van Hamme supercongruence original proof methods include the WZ method used by Zudilin and Osburn to prove both (B.2) and (K.2).  The WZ method of proof looks inviting for generalization because it relies on the Wilf-Zeilberger algorithm \cite{AeqB}.  In these proofs, a clever choice of factors leads the WZ algorithm to seemingly do much of the creative work needed to establish the desired supercongruences.  An obstacle to the WZ method is that, to the author's knowledge, finding the necessary factors has relied on ad hoc methods.  To help address this shortcoming, this paper demonstrates a procedure to pick factor candidates and provides uniform proofs for (B.2), (C.2), (D.2), (E.2), (F.2), (G.2), and (H.2).  As will be shown, the WZ method tends to couple well with the analysis of the $p$-adic gamma function by Long and Ramakrishna \cite{LongRama} in this context.  In fact, the proofs of (G.2) and (H.2) given here extend the Van Hamme supercongruences modulo higher powers of $p$, recovering results by Swisher \cite{Swisher15} and Liu \cite{Liu}, respectively.  In addition, the methods presented here show that (I.2) can be proved with the WZ method, yet without any clever factors at all.  Together with the (K.2) proof by Osburn and Zudilin, this shows that at least nine out of the original thirteen Van Hamme supercongruences can be proved using the WZ method.

To state this paper's main result, a few definitions are in order.  Per Petkov\v{s}ek, Wilf, and Zeilberger \cite{AeqB}, a function $F(n)$ is called a \emph{hypergeometric term} if $F(n+1) / F(n)$ is a rational function of $n$.  Similarly, a function $F(n, k)$ is a hypergeometric term in both $n$ and $k$ if $F(n+1, k) / F(n, k)$ and $F(n, k+1) / F(n, k)$ are rational in $n$ and $k$, respectively.  For brevity, a function may be said to be hypergeometric in the variables for which it is a hypergeometric term.

The following result is key to the approach presented here.

\begin{theorem} \label{t-wz-pair-polynomial-elimination}
Take a field $\mathbb{F}$ of characteristic zero, and let $F(n, k), G(n, k) : \Z^2 \to \mathbb{F}$ be hypergeometric in both $n$ and $k$.  Suppose that for some polynomials $p_0, p_1 \in \mathbb{F}[k]$ one has that
$$p_1(k) F(n, k+1) + p_0(k) F(n, k) = G(n+1, k) - G(n, k)$$

\noindent
holds for $0 \leq k \leq m$.  In addition, suppose $p_0, p_1$ are not zero for any such $k$, and that they split into linear factors over $\mathbb{F}$.  Then there exist functions $\Fbar (n, k), \Gbar (n, k)$, hypergeometric in both $n$ and $k$, for which $\Fbar (n, 0) = F(n, 0)$ and such that
$$\Fbar (n, k + 1) - \Fbar (n, k) = \Gbar (n+1, k) - \Gbar (n, k),$$

\noindent
that is, the functions $\Fbar (n, k)$ and $\Gbar (n, k)$ form a WZ pair.  Moreover, the proof is constructive and so $\Fbar (n, k)$ and $\Gbar (n, k)$ can be defined explicitly (see Section \ref{s-delta-constant-coefficients}).
\end{theorem}

Using Theorem \ref{t-wz-pair-polynomial-elimination} yields the theorem below.

\begin{theorem} \label{t-van-hamme}
The Van Hamme supercongruences (B.2)-(H.2) listed in Table \ref{table-van-hamme} can be proved by a streamlined WZ method using Theorem \ref{t-wz-pair-polynomial-elimination} and Theorem \ref{t-gk-theorems} of Long and Ramakrishna.  In particular, the supercongruence (G.2) holds modulo $p^4$ (as shown in Swisher \cite{Swisher15}), and (H.2) can be extended modulo $p^3$ when $p \equiv 3 \pmod{4}$ (as shown in Liu \cite{Liu}).
\end{theorem}

In addition, the proposition below holds.

\begin{proposition} \label{p-van-hamme-i2}
The Van Hamme supercongruence (I.2) follows directly from an application of Gosper's algorithm, together with well known congruences for Pochhammer symbols.
\end{proposition}

\begin{table}[ht]
\caption{Van Hamme supercongruences, as originally stated for primes $p$, proved in Theorem \ref{t-van-hamme} and Proposition \ref{p-van-hamme-i2}.} \label{table-van-hamme}
\begin{tabular}{rrcll}
(B.2) \quad & $\sum_{n=0}^{(p-1)/2} (4n+1) (-1)^n \cdot \frac{\ph{1/2}{n}^3}{\ph{1}{n}^3}$ & $\equiv$ & $\frac{-p}{\Gamma_p(1/2)^2} \pmod{p^3}$ & $p > 2$ \\
\hlineFix
(C.2) \quad & $\sum_{n=0}^{(p-1)/2} (4n+1) \cdot \frac{\ph{1/2}{n}^4}{\ph{1}{n}^4}$ & $\equiv$ & $p \pmod{p^3}$ & $p > 2$ \\
\hlineFix
(D.2) \quad & $\sum_{n=0}^{(p-1)/3} (6n+1) \cdot \frac{\ph{1/3}{n}^6}{\ph{1}{n}^6}$ & $\equiv$ & $-p \Gamma_p(1/3)^9 \pmod{p^4}$ & $p \equiv 1 \pmod{6}$ \\
\hlineFix
(E.2) \quad & $\sum_{n=0}^{(p-1)/3} (6n+1) (-1)^n \cdot \frac{\ph{1/3}{n}^3}{\ph{1}{n}^3}$ & $\equiv$ & $p \pmod{p^3}$ & $p \equiv 1 \pmod{6}$ \\
\hlineFix
(F.2) \quad & $\sum_{n=0}^{(p-1)/4} (8n+1) (-1)^n \cdot \frac{\ph{1/4}{n}^3}{\ph{1}{n}^3}$ & $\equiv$ & $\frac{-p}{\Gamma_p(1/4) \Gamma_p(3/4)} \pmod{p^3}$ & $p \equiv 1 \pmod{4}$ \\
\hlineFix
(G.2) \quad & $\sum_{n=0}^{(p-1)/4} (8n+1) \cdot \frac{\ph{1/4}{n}^4}{\ph{1}{n}^4}$ & $\equiv$ & $\frac{p \Gamma_p(1/2) \Gamma_p(1/4)}{\Gamma_p(3/4)} \pmod{p^3}$ & $p \equiv 1 \pmod{4}$ \\
\hlineFix
(H.2) \quad & $\sum_{n=0}^{(p-1)/2} \frac{\ph{1/2}{n}^3}{\ph{1}{n}^3}$ & $\equiv$ & $\left \{ \begin{matrix} - \Gamma_p(1/4)^4 \pmod{p^2} \\ \hlineFix 0 \pmod{p^2} \end{matrix} \right. $ & $\begin{matrix} p \equiv 1 \pmod{4} \\ \hlineFix p \equiv 3 \pmod{4} \end{matrix}$ \\
\hlineFix
(I.2) \quad & $\sum_{n=0}^{(p-1)/2} \frac{\ph{1/2}{n}^2}{\ph{1}{n}^2 (n+1)}$ & $\equiv$ & $2p^2 \pmod{p^3}$ & $p > 2$
\end{tabular}
\end{table}

Many authors have extended the original Van Hamme supercongruences to hold modulo higher powers of $p$, see for example Swisher \cite{Swisher15}.  Liu extended (H.2) for primes $p \equiv 3 \pmod{4}$ \cite{Liu}.  Recently, Guo and Wang used the WZ method to prove (E.2) and (F.2) \cite{GuoWang-E2F2}.  The proof given here is different, and in some ways simpler because the right hand side of the WZ pair telescopes to zero.  That is, the proofs here are similar to Zudilin's proof of (B.2) \cite{Zudilin09}, rather than to Osburn and Zudilin's proof of (K.2) \cite{OsburnZudilin15} which telescopes the left hand side.  In addition, the present proofs of (B.2)-(H.2) have the same shape.  However, Guo and Wang's proof \cite{GuoWang-E2F2} is more general and applies to more cases.  In another recent development, Jana and Karmakar prove conjectures of Guo related to (B.2) and (C.2) using the WZ method and a parametrized WZ pair \cite{JanaKarmakar}.  In their work, they state that despite this generalization the WZ method still requires an inspired guess.  In contrast, the approach presented here introduces the notion of a \emph{WZ device} with the intent of mechanically recovering a suitable WZ pair from the output of the WZ algorithm.  This approach is somewhat similar to that of Guillera for infinite series \cite{Guillera19}.

Supercongruences can have so called $q$-analogues, in which $q$-Pochhammer symbols substitute for Pochhammer symbols.  Studying a $q$-analogue may lead to a proof of the original supercongruence by means of $q$-microscoping, see Guo and Zudilin \cite{GuoZudilinMicroscope}.  Several authors, including Guo, He, Liu, and Wang, gave numerous Van Hamme supercongruence $q$-analogues (see the following non exhaustive list for examples: \cite{qa-A2, qa-B2, qa-C2, qa-D2, qa-G2}).  Another general approach is that of Beukers \cite{Beukers}, who uses modular forms to prove several families of supercongruences at once.  In comparison, the focus here is on extending the applicability and usability of the WZ method to the original supercongruences.

The remainder of the paper is organized as follows.  Section \ref{s-background} summarizes background material and auxiliary results.  This section also shows the Van Hamme supercongruence (I.2) can be proved as a special case of the WZ method.  Section \ref{s-delta-constant-coefficients} develops a procedure to find candidates for the clever factors needed for the WZ method to succeed, and exemplifies it to prove (H.2).  In Section \ref{s-applications}, the preceding results together with an important remark are used to prove the supercongruences (B.2), (C.2), (E.2), (F.2), and (G.2) with a uniform proof structure.  In particular, this method yields the supercongruence (G.2) modulo $p^4$ as shown in Swisher \cite{Swisher15}, and (H.2) modulo $p^3$ when $p \equiv 3 \pmod{4}$ as shown in Liu \cite{Liu}.

To develop these ideas further, Section \ref{s-wz-devices} introduces the notion of a \emph{WZ device}.  This section also proves (D.2), completing the proof of Theorem \ref{t-van-hamme}.  Each of the proofs given in Sections 3--5 utilizes a WZ device.  A list of WZ devices used here is provided at the end of the paper in Table \ref{table-wz-devices}.  The proofs for (C.2) and (D.2) suggest they hold modulo higher powers of $p$, as noted by Swisher \cite{Swisher15} and Long and Ramakrishna \cite{LongRama}, respectively.  An argument in terms of the present approach remains elusive.

Some of the symbolic computations, as well as most of the computational work related to the WZ algorithm, Gosper's algorithm, polynomial factoring, and solving simultaneous polynomial equations, were performed using Maple.

\subsubsection*{Acknowledgments}  Many thanks go to this author's advisor, Prof.~Holly Swisher, for supervising the thesis work which formed the basis of this paper.  This work was supported in part by NSF Grant DMS-2101906 (PI: Holly Swisher).

\section{Background and auxiliary results} \label{s-background}

The present paper uses and extends the following notation from Concrete Mathematics \cite{Concrete}.  For $a, b \in \Z$, writing $a \perp b$ will denote $\gcd(a, b) = 1$.  Moreover, if $x \in \Q_p$ then $x \perp p$ will denote that $x \in \Z_p^\times$.  In particular, if $x = a/b \in \Q$ with $a \perp b$, then $x \perp p$ when both $a \perp p$ and $b \perp p$.  Also, Pochhammer symbols with negative indices $-k$ are defined in \cite[ex.~2.9]{Concrete} by
$$\ph{x}{-k} = \frac{1}{\ph{x-k}{k}}.$$

\noindent
Specifically, for integers $k > 0$ it follows that
\begin{equation} \label{e-self-nulling}
\frac{1}{\ph{1}{-k}} = \ph{1-k}{k} = 0.
\end{equation}

\subsection{Generalized WZ pairs}

The WZ method relies on finding a so-called \emph{WZ pair} so that a sum can be evaluated by telescoping.  Here, the definitions of these pairs is relaxed as follows.  To begin, define a shift operator $N$ to act on a function $F(n)$ in the variable $n$ by
$$NF(n) = F(n+1).$$

\noindent
By convention, in this paper the shift operators $N, K$ act on the variables $n, k$, respectively.  Shift operators act on multivariate functions by holding all but one variable constant.  For example,
$$NF(n, k) = F(n+1, k).$$

A polynomial $\Delta$ on a shift operator $K$ of the form
\begin{equation} \label{e-Delta-shape}
\Delta = \sum_{j=0}^d p_j(k) K^j,
\end{equation}

\noindent
where each $p_j(k)$ is a polynomial in $k$, will be referred to as a \emph{difference operator} as in Petkov\v{s}ek, Wilf, and Zeilberger \cite{AeqB}.  For simplicity, it may be said that $\Delta$ acts on the variable $k$.  The way $\Delta$ applies to a function $F$ is defined linearly in terms of the action of the corresponding shift operator on $F$.  For instance, if the shift operator of $\Delta$ is $K$, then
$$\Delta F(n, k) = \sum_{j=0}^d p_j(k) F(n, k+j).$$

\noindent
Here, by convention, the difference operator $\Delta$ always acts on $k$ and its shift operator is always $K$.

In the present work, the functions $F(n, k)$ and $G(n, k)$ form a \emph{generalized WZ pair} if there is some difference operator $\Delta \neq 0$ acting on $k$ such that
\begin{equation} \label{e-gen-wz-pair}
\Delta F(n, k) = (N-1)G(n, k) = G(n+1, k) - G(n, k)
\end{equation}

\noindent
holds.  In the literature, the term \emph{WZ pair} refers exclusively to the case where $\Delta = K-1$ as in Petkov\v{s}ek, Wilf, and Zeilberger \cite{AeqB}, namely when
$$F(n, k+1) - F(n, k) = G(n+1, k) - G(n, k).$$

\noindent
Note that WZ pairs may be presented differently, sometimes due to a different choice of variable names, and sometimes due to a change of variables.  See for example Zudilin \cite{Zudilin09}, and Osburn and Zudilin \cite{OsburnZudilin15}.

Recall that the WZ algorithm of Wilf and Zeilberger \cite{AeqB} takes a function $F(n, k)$, hypergeometric in both $n$ and $k$, and returns another function $G(n, k)$ hypergeometric in both $n$ and $k$, as well as a difference operator $\Delta$ acting on $k$, which altogether satisfy \eqref{e-gen-wz-pair}.  The algorithm is guaranteed to work for a large class of functions $F(n, k)$ called \emph{proper terms}.  Simply put, the WZ algorithm succeeds by producing a generalized WZ pair for the function $F(n, k)$.

Observe that if $\Delta$ is of the form
\begin{equation} \label{e-Delta-linear-shape}
\Delta = p_1(k) K + p_0(k),
\end{equation}

\noindent
that is $d = 1$ in \eqref{e-Delta-shape}, then \eqref{e-gen-wz-pair} gives that $F(n, k)$ and $G(n, k)$ satisfy the hypothesis of Theorem \ref{t-wz-pair-polynomial-elimination}.

The WZ algorithm relies on Gosper's algorithm \cite{AeqB}, which takes a hypergeometric term $F(n)$ and either returns a hypergeometric term $G(n)$ such that $F(n) = (N-1)G(n)$, or proves that no such $G(n)$ exists.  When Gosper's algorithm succeeds, sums involving $F(n)$ reduce to telescoping.  Gosper's algorithm fails for all the Van Hamme supercongruence sum terms except for that of (I.2).

\begin{proof}[Proof of Proposition \ref{p-van-hamme-i2}]
Running Gosper's algorithm on $F(n) = \ph{1/2}{n}^2 / (\ph{1}{n}^2 (n+1))$ returns
$$G(n) = \frac{4n \ph{1/2}{n}^2}{\ph{1}{n}^2},$$

\noindent
satisfying $F(n) = (N-1) G(n) = G(n+1) - G(n)$.  Consequently, by telescoping,
$$\sum_{n=0}^{\fpmoh} \frac{\ph{1/2}{n}^2}{\ph{1}{n}^2 (n+1)}
= \sum_{n=0}^{\fpmoh} \left ( G(n+1) - G(n) \right )
= G \left ( \fppoh \right ) - G(0).$$

\noindent
Since $G(0) = 0$, the supercongruence depends on analyzing
\begin{equation} \label{e-i2}
\sum_{n=0}^{\fpmoh} \frac{\ph{1/2}{n}^2}{\ph{1}{n}^2 (n+1)}
= \frac{4 \cdot \fppoh \cdot \ph{1/2}{\fppoh}^2}{\ph{1}{\fppoh}^2}
\end{equation}

\noindent
modulo $p^3$.  In contrast to this mechanical proof, Van Hamme proves \eqref{e-i2} by induction \cite{VanHamme97}.  Past this point, the proof of (I.2) can continue as in Van Hamme's.  Generally, this involves using well known congruence results for Pochhammer symbols.
\end{proof}

The preceding shows (I.2) is a special case of the WZ method where the underlying Gosper's algorithm succeeds in finding a closed form for the sum.  Alternatively, from the WZ method perspective, here the WZ algorithm returns the trivial difference operator $\Delta = 1$.  This can be seen clearly by running the WZ algorithm on $F(n, k) = F(n) \cdot \ph{0}{k}$ then setting $k = 0$.

\subsection{Long-Ramakrishna approximation for $\Gamma_p$}

Following Koblitz \citep{Koblitz}{89--91}, this section introduces Morita's $p$-adic gamma function $\Gamma_p : \Z \to \Z$, defined for odd primes $p$ by
$$\Gamma_p(k) = (-1)^k \prod_{\substack{0 < j < k \\ j \perp p}} j.$$

\noindent
Assuming the empty product is $1$, one has the special cases $\Gamma_p(0) = 1$ and $\Gamma_p(1) = -1$.  Since $\Z$ is dense in $\Z_p$, this function extends to a unique continuous function $\Gamma_p : \Z_p \to \Z_p^\times$ defined by
$$\Gamma_p(s) = \lim_{\substack{k \to s \\ k \in \N}} (-1)^k \prod_{\substack{0 < j < k \\ j \perp p}} j.$$

The critical property for the extension to hold is that for every $\varepsilon > 0$, there is some integer $M$ so that if $s \equiv s' \pmod{p^M}$, then $|\Gamma_p(s) - \Gamma_p(s')| < \varepsilon$.   This follows because $\Gamma_p(k) \equiv \Gamma_p(k') \pmod{p^M}$ if $k \equiv k' \pmod{p^M}$ per Koblitz \cite{Koblitz}.

The $\Gamma_p$ function satisfies the following properties \cite{Koblitz}.  First, for $s \in \Z_p$,
\begin{equation} \label{e-koblitz-props-1}
\frac{\Gamma_p(s+1)}{\Gamma_p(s)} = \begin{cases}
-s & s \in \Z_p^\times \\
-1 & s \notin \Z_p^\times
\end{cases}.
\end{equation}

\noindent
In addition, for $x \in \Z_p$,
\begin{equation} \label{e-gamma-p-reflection}
\Gamma_p(x) \Gamma_p(1-x) = (-1)^{a_0(x)},
\end{equation}

\noindent
where $1 \leq a_0(x) \leq p$ is the smallest positive residue of $x$ modulo $p$.

Applying \eqref{e-koblitz-props-1} inductively, and as described in Swisher \cite{Swisher15}, leads to additional $\Gamma_p$ properties.  First, for integer $k \geq 0$ and $\alpha \in \Q_p$, if $\ph{\alpha}{k} \perp p$ then
\begin{equation} \label{e-gamma-p-props-1}
\frac{\Gamma_p(\alpha + k)}{\Gamma_p(\alpha)} = (-1)^k \ph{\alpha}{k}.
\end{equation}

\noindent
In general,
\begin{equation} \label{e-gamma-p-props-2}
\frac{\Gamma_p(\alpha + k)}{\Gamma_p(\alpha)} = (-1)^k \prod_{\substack{j=0 \\ \alpha+j \perp p}}^{k-1} (\alpha + j).
\end{equation}

\noindent
Recalling $p$ is odd, one has the special case
\begin{equation} \label{e-gamma-p-props-3}
\Gamma_p(1/2)^2 = (-1)^{\frac{p+1}{2}}.
\end{equation}

In addition, Long and Ramakrishna \cite{LongRama} prove the following results for the logarithmic derivative of $\Gamma_p$.  For integer $k \geq 0$ and $a \in \Z_p$, the $k$-th logarithmic derivative of $\Gamma_p(a)$ is defined by
$$G_k(a) = \frac{\Gamma^{(k)}_p(a)}{\Gamma_p(a)}.$$

\noindent
The $k$-th logarithmic derivative $G_k(a)$ satisfies the following identities:
\begin{equation} \label{e-gk-easy-identities}
G_0(a) = 1, \qquad G_1(a) = G_1(1-a), \qquad G_2(a) + G_2(1-a) = 2 G_1(a)^2.
\end{equation}

\noindent
Moreover, for primes $p \geq 5$,
\begin{equation} \label{e-g1-g2-0}
G_1(0)^2 = G_2(0).
\end{equation}

\noindent
In fact, putting together \eqref{e-g1-g2-0} and \eqref{e-gk-easy-identities} with $a = 0$ yields
$$G_2(0) = G_2(1).$$

The next theorem is a key observation about congruences involving $\Gamma_p$, and combines the results in Long and Ramakrishna \cite[Thm.~14]{LongRama}, and Swisher \cite[Eqn.~(15)]{Swisher15}.

\begin{theorem} \cite{LongRama, Swisher15} \label{t-gk-theorems}
For primes $p \geq 5$, positive integer $r$, with $a, b \in \Q \cap \Z_p$, and integer $0 \leq t \leq 2$,
$$\frac{\Gamma_p(a+bp^r)}{\Gamma_p(a)} \equiv \sum_{k=0}^t \frac{G_k(a)}{k!} \cdot (bp^r)^k \pmod{p^{(t+1)r}}.$$

\noindent
The congruence holds for $t = 3$ if $r = 1$ and $p > 5$, and for $t = 4$ if $p \geq 11$.
\end{theorem}

It follows directly from Theorem \ref{t-gk-theorems} that, for $r, s \in \Q \cap \Z_p$ and primes $p \geq 5$, \begin{align}
\label{e-gk-binomial-2} \left ( 1 + G_1(s) rp + G_2(s) \cdot \frac{r^2 p^2}{2} \right )^2 
& \equiv & 1 + G_1(s) 2rp + G_1(s)^2 r^2 p^2 + G_2(s) r^2 p^2 \pmod{p^3}, \\
\label{e-gk-binomial-4} \left ( 1 + G_1(s) rp + G_2(s) \cdot \frac{r^2 p^2}{2} \right )^4
& \equiv & 1 + G_1(s) 4rp + G_1(s)^2 6 r^2 p^2 + G_2(s) 2r^2 p^2 \pmod{p^3}.
\end{align}

\subsection{Auxiliary results}

For future use, note that if $p$ is prime, $m \in \Z$, and $n \in \N$, then
\begin{equation} \label{e-mod-trick}
p^{n-1} (1 + pm) \equiv p^{n-1} \pmod{p^n}.
\end{equation}

Furthermore, observe the following identities between Pochhammer symbols and ratios of $\Gamma_p$ function evaluations.

\begin{lemma} \label{l-ph-gp-identities-H}
The following identities hold for odd primes $p$:
\begin{eqnarray*}
\ph{1/2}{p-1} & = & \frac{p}{2} \cdot \frac{\Gamma_p(1/2 + p - 1)}{\Gamma_p(1/2)}
= \frac{-p}{2p-1} \cdot \frac{\Gamma_p(1/2 + p)}{\Gamma_p(1/2)}, \\
\ph{1}{\fpmoh} & = & (-1)^{\fppoh} \Gamma_p \left ( \fppoh \right ), \\
\ph{1}{\fpmoh}^2 & = & \Gamma_p \left ( \fppoh \right )^2, \\
\ph{1/4}{\fpmoh} & = & \begin{cases}
\frac{p}{4} \cdot \frac{\Gamma_p \left ( \frac{-1}{4} + \frac{p}{2} \right )}{\Gamma_p(1/4)} & p \equiv 1 \pmod{4} \\
- \frac{\Gamma_p \left ( \frac{-1}{4} + \frac{p}{2} \right )}{\Gamma_p(1/4)} & p \equiv 3 \pmod{4}
\end{cases}, \\
\ph{3/4}{\fpmoh} & = & \begin{cases}
\frac{\Gamma_p \left ( \frac{1}{4} + \frac{p}{2} \right )}{\Gamma_p(3/4)} & p \equiv 1 \pmod{4} \\
- \frac{p}{4} \cdot \frac{\Gamma_p \left ( \frac{1}{4} + \frac{p}{2} \right )}{\Gamma_p(3/4)} & p \equiv 3 \pmod{4}
\end{cases}.
\end{eqnarray*}
\end{lemma}

\begin{proof}
The term $\ph{1/2}{p-1}$ has only one factor of $p$, namely $p/2$, because for $0 \leq k < p-1$,
$$\frac{1}{2} \leq \frac{1 + 2k}{2} \leq \frac{2p-3}{2},$$

\noindent
and $p$ divides $1 + 2k$ only when $k = \dpmoh$ and so $1 + 2k = p$.  Hence, it follows from \eqref{e-gamma-p-props-2} that
$$\ph{1/2}{p-1} = \frac{p}{2} \cdot \frac{\Gamma_p(1/2 + p - 1)}{\Gamma_p(1/2)}.$$

\noindent
By \eqref{e-koblitz-props-1},
$$\Gamma_p(1/2 + p - 1) = \frac{-2}{2p-1} \cdot \Gamma_p(1/2 + p).$$

\noindent
Putting these together,
$$\ph{1/2}{p-1} = \frac{-p}{2p-1} \cdot \frac{\Gamma_p(1/2 + p)}{\Gamma_p(1/2)},$$

\noindent
proving the first identity.

Next, observe that since $\dpmoh < p$ and $\Gamma_p(1) = -1$, by \eqref{e-gamma-p-props-1} one has that
$$\ph{1}{\fpmoh} = (-1)^{\fpmoh} \cdot \frac{\Gamma_p \left (1 + \fpmoh \right )}{\Gamma_p(1)} = (-1)^{\fppoh} \Gamma_p \left ( \fppoh \right ).$$

\noindent
Squaring both sides, it follows that
$$\ph{1}{\fpmoh}^2 = \Gamma_p \left ( \fppoh \right )^2.$$

Now, suppose $p \equiv 1 \pmod{4}$.  Then $\dpmoh$ is even and $\ph{1/4}{\fpmoh}$ has a factor of $p/4$.  Hence, adding the missing factor $p/4$ to $\Gamma_p(1/4+\fpmoh)$ and using \eqref{e-gamma-p-props-2},
$$\ph{1/4}{\fpmoh}
= \frac{p}{4} \cdot (-1)^{\fpmoh} \cdot \frac{\Gamma_p \left(1/4 + \fpmoh \right )}{\Gamma_p(1/4)}
= \frac{p}{4} \cdot \frac{\Gamma_p \left ( \frac{-1}{4} + \frac{p}{2} \right )}{\Gamma_p(1/4)}.$$

\noindent
However, if $p \equiv 3 \pmod{4}$, then $\dpmoh$ is odd and $\ph{1/4}{\fpmoh} \perp p$.  A similar argument using \eqref{e-gamma-p-props-2} shows
$$\ph{1/4}{\fpmoh}
= (-1)^{\fpmoh} \cdot \frac{\Gamma_p \left (1/4 + \fpmoh \right )}{\Gamma_p(1/4)}
= - \frac{\Gamma_p \left ( \frac{-1}{4} + \frac{p}{2} \right )}{\Gamma_p(1/4)}.$$

Finally, if $p \equiv 1 \pmod{4}$ then $\dpmoh$ is even and $\ph{3/4}{\fpmoh} \perp p$.  Consequently, using \eqref{e-gamma-p-props-1},
$$\ph{3/4}{\fpmoh}
= (-1)^{\fpmoh} \cdot \frac{\Gamma_p \left ( 3/4 + \frac{p-1}{2} \right )}{\Gamma_p(3/4)}
= \frac{\Gamma_p \left ( \frac{1}{4} + \frac{p}{2} \right )}{\Gamma_p(3/4)}.$$

\noindent
Otherwise $p \equiv 3 \pmod{4}$, whence $\dpmoh$ is odd and $\ph{3/4}{\fpmoh}$ has a factor of $p/4$.  Using \eqref{e-gamma-p-props-2} one last time,
$$\ph{3/4}{\fpmoh}
= (-1)^{\fpmoh} \cdot \frac{p}{4} \cdot \frac{\Gamma_p \left ( 3/4 + \frac{p-1}{2} \right )}{\Gamma_p(3/4)}
= - \frac{p}{4} \cdot \frac{\Gamma_p \left ( \frac{1}{4} + \frac{p}{2} \right )}{\Gamma_p(3/4)}.$$

\noindent
This completes the proof.
\end{proof}

The following congruence will be useful later.

\begin{lemma} \label{l-h2-gp-quotient-p3}
For primes $p \geq 5$,
$$\frac{\Gamma_p(1/2 + p)}{\Gamma_p(1/2)} \cdot \frac{\Gamma_p(1/2)^2}{\Gamma_p \left ( \frac{1}{2} + \frac{p}{2} \right )^2} \equiv 1 \pmod{p^3}.$$
\end{lemma}

\begin{proof}
Rewrite the left hand side factors modulo $p^3$ using Theorem \ref{t-gk-theorems} to obtain
\begin{eqnarray*}
\frac{\Gamma_p(1/2 + p)}{\Gamma_p(1/2)} & \equiv & 1 + G_1(1/2)p + \frac{G_2(1/2)p^2}{2} \pmod{p^3}, \\
\frac{\Gamma_p(1/2)^2}{\Gamma_p \left ( \frac{1}{2} + \frac{p}{2} \right )^2} & \equiv & \frac{1}{\left ( 1 + \frac{G_1(1/2)p}{2} + \frac{G_2(1/2)p^2}{8} \right )^2} \pmod{p^3}.
\end{eqnarray*}

\noindent
Using \eqref{e-gk-binomial-2} and \eqref{e-gk-easy-identities},
\begin{eqnarray*}
\left ( 1 + \frac{G_1(1/2)p}{2} + \frac{G_2(1/2)p^2}{8} \right )^2
& \equiv & 1 + G_1(1/2)p + \frac{G_1(1/2)^2 p^2}{4} + \frac{G_2(1/2)p^2}{4} \\
& \equiv & 1 + G_1(1/2)p + \frac{G_2(1/2)p^2}{2} \pmod{p^3}.
\end{eqnarray*}

\noindent
The denominator above does not have factors of $p$.  Thus, these two factors cancel modulo $p^3$, as desired.
\end{proof}

Due to the symmetry illustrated by \eqref{e-gamma-p-reflection}, the following proposition holds as well.

\begin{lemma} \label{l-gamma-p-a0-reflection}
If $1 \leq a_0(x) \leq p$ is the smallest positive residue of $x \in \Z_p$ modulo $p > 2$, then
$$a_0(x) \equiv a_0(1-x) \pmod{2}.$$
\end{lemma}

\begin{proof}
For a fixed $x$, applying \eqref{e-gamma-p-reflection} to both $x$ and $1-x$ yields that
$$(-1)^{a_0(x)} = (-1)^{a_0(1-x)},$$

\noindent
whence $a_0(x)$ and $a_0(1-x)$ have the same parity.
\end{proof}

Finally, suppose the prime $p$ is of the form $6p' + 1$ for $p' \in \Z$, and that $1 \leq a_0(2/3) \leq p$ is the smallest positive residue of $2/3$ modulo $p$.  Then $3(2p'+1) \equiv 6p' + 3 \equiv 2 \pmod{p}$ implies
\begin{equation} \label{e-a0-2o3-1mod6}
(-1)^{a_0(2/3)} = (-1)^{2p' + 1} = -1.
\end{equation}

\section{Linear difference operators with constant coefficients} \label{s-delta-constant-coefficients}

Generally speaking, the WZ method of proving supercongruences as used by Zudilin \cite{Zudilin09}, Osburn and Zudilin \cite{OsburnZudilin15}, Mao and Wen \cite{MaoWen19}, and other authors, consists of creating a generalized WZ pair for $F(n, k)$ such that two things happen simultaneously.  First, running the WZ algorithm on $F(n, k)$ returns the difference operator $\Delta = K - 1$.  Second, a congruence analysis on both $F(n, k)$ and the $G(n, k)$ function returned by the WZ algorithm is tractable.

Suppose a supercongruence's sum has a summand denoted $F(n)$ that is hypergeometric in $n$.  If the sum has a closed form in hypergeometric terms, Gosper's algorithm finds $G(n)$ hypergeometric in $n$ such that $F(n) = (N-1)G(n) = G(n+1) - G(n)$ and the congruence analysis can proceed from the resulting telescoping sum:
$$\sum_{n=0}^M F(n) = G(M+1) - G(0).$$

\noindent
Otherwise, to use the WZ algorithm, the term $F(n)$ must be modified so that it also depends on the variable $k$.  This must be done in a way that breaks symmetry between $n$ and $k$, for otherwise one runs the risk of running Gosper's algorithm in disguise.  Moreover, it must be possible to recover the original $F(n)$ from $F(n, k)$.  Suppose then that $F(n, k)$ is the modified version of $F(n)$.  To facilitate the subsequent congruence analysis, Zudilin \cite{Zudilin09} and others have conveniently relied on \emph{self nulling} factors such as $\ph{1}{n-k}$ in the denominator of $F(n, k)$ so that $F(n, k) = 0$ when $k > n$ by \eqref{e-self-nulling}.  Changing one of the Pochhammer symbols in the numerator from $\ph{x}{n}$ to $\ph{x}{n+k}$ in tandem tends to help the WZ algorithm keep the difference operators simple.

For example, the left hand side of the Van Hamme supercongruence (H.2) is
$$\sum_{n=0}^{(p-1)/2} \frac{\ph{1/2}{n}^3}{\ph{1}{n}^3}.$$

\noindent
Set $F(n) = \ph{1/2}{n}^3 / \ph{1}{n}^3$.  Since Gosper's algorithm fails on $F(n)$, generalize the modifications used by Zudilin's proof of (B.2) \cite{Zudilin09} and modify $F(n)$ as follows:
\begin{equation} \label{e-wz-device-h2}
F(n) \mapsto F(n, k) = \frac{\ph{1/2}{n}^2 \ph{1/2}{n+k}}{\ph{1}{n}^2 \ph{1}{n-k}}.
\end{equation}

\noindent
It is clear that $F(n, 0)$ recovers the original $F(n)$ summand.  At this point, applying the WZ algorithm to $F(n, k)$ returns the difference operator
\begin{equation} \label{e-bad-telescoping}
\Delta = -4(4k + 3)K - (4k + 1)(2k + 1)^2.
\end{equation}

\noindent
Regardless of what $G(n, k)$ is, the presence of non constant polynomial coefficients in the difference operator suggests moving further will be much harder than if $\Delta = K - 1$.  Even worse, the coefficients' signs will greatly obstruct telescoping.  However, both polynomial coefficients split into plain linear factors with integer coefficients, motivating Theorem \ref{t-wz-pair-polynomial-elimination} proved below.

\begin{proof}[Proof of Theorem \ref{t-wz-pair-polynomial-elimination}]
Let $\Delta$ be the difference operator of the form $\Delta = p_1(k) K + p_0(k)$ as in \eqref{e-Delta-linear-shape}.  For $0 \leq k \leq m$, the hypotheses provide that
\begin{equation} \label{e-t1-1}
\Delta F(n, k) = p_1(k) F(n, k+1) + p_0(k) F(n, k) = G(n+1, k) - G(n, k),
\end{equation}

\noindent
whence $F(n, k)$ and $G(n, k)$ form a generalized WZ pair.  If there exists $q(k)$, hypergeometric in $k$, with $q(0) = 1$, and satisfying
\begin{equation} \label{e-t1-q}
\frac{q(k+1)}{q(k)} = \frac{-p_1(k)}{p_0(k)},
\end{equation}

\noindent
then setting
$$\Fbar (n, k) = q(k) F(n, k)
\qquad \textrm{and} \qquad
\Gbar (n, k) = -q(k) G(n, k) / p_0(k)$$

\noindent
suffices.  In fact, multiplying \eqref{e-t1-1} by $-q(k) / p_0(k)$ gives
$$\Fbar (n, k+1) - \Fbar (n, k) = \Gbar (n + 1, k) - \Gbar (n, k),$$

\noindent
whence $\Fbar (n, k)$ and $\Gbar (n, k)$ form a WZ pair with $\Delta = K - 1$.  As $q(k)$, $F(n, k)$, and $G(n, k)$ are hypergeometric in $n$ and $k$, the functions $\Fbar (n, k)$ and $\Gbar (n, k)$ are hypergeometric in $n$ and $k$ too.  Moreover, as $q(0) = 1$ then $\Fbar (n, 0) = F(n, 0)$ follows.

Now, suppose that $r \in \mathbb{F}[k]$ with $\deg r = d \geq 0$ splits into linear factors over $\mathbb{F}$ such that
$$r(k) = c \cdot \prod_{j=1}^d (m_j k + b_j),$$

\noindent
where $m_j \neq 0$ for any $j$.  For each such $r(k)$, define
\begin{equation} \label{e-thm-transform}
\varphi_r(k) = c^k \cdot \prod_{j=1}^d m_j^k \ph{b_j / m_j}{k}.
\end{equation}

\noindent
Then, setting
\begin{equation} \label{e-t1-qk}
q(k) = (-1)^k \cdot \frac{\varphi_{p_1}(k)}{\varphi_{p_0}(k)}
\end{equation}

\noindent
gives that \eqref{e-t1-q} is satisfied, that $q(k)$ is hypergeometric in $k$, and that $q(0) = 1$.  The restriction that $p_0(k), p_1(k) \neq 0$ for any relevant value of $k$ prevents both division by zero in \eqref{e-t1-q}, as well as multiplication by zero in \eqref{e-t1-1} and \eqref{e-thm-transform}.  This completes the proof.
\end{proof}

\subsection{Application of Theorem \ref{t-wz-pair-polynomial-elimination} to (H.2)}

To illustrate the utility of Theorem \ref{t-wz-pair-polynomial-elimination}, consider the following proof of the Van Hamme supercongruence (H.2).  Unlike the proof in Swisher \cite{Swisher15}, say, the argument here uses the WZ method to reach the congruence analysis stage and does not require the use of a specialized hypergeometric transformation.

To prevent clutter, for $d \in \N$ define
$$\spmox{d} = \fpmox{d}.$$

\begin{proof}[Proof of Theorem \ref{t-van-hamme}, (H.2)]
The following extends (H.2) modulo $p^3$ for the case $p \equiv 3 \pmod{4}$ as shown in Liu \cite{Liu}.  Set $F(n) = \ph{1/2}{n}^3 / \ph{1}{n}^3$, then apply \eqref{e-wz-device-h2} to obtain
$$F(n, k) = \frac{\ph{1/2}{n}^2 \ph{1/2}{n+k}}{\ph{1}{n}^2 \ph{1}{n-k}}.$$

\noindent
Running the WZ algorithm on $F(n, k)$ gives a difference operator $\Delta$ of the form \eqref{e-Delta-shape} with polynomial coefficients
\begin{eqnarray*}
p_1(k) & = & -4(4k + 3), \\
p_0(k) & = & - (4k+1) (2k+1)^2,
\end{eqnarray*}

\noindent
which is unsuitable for further analysis as discussed in regards to \eqref{e-bad-telescoping}.  Thus, the hypergeometric $G(n, k)$ satisfying $\Delta F(n, k) = (N-1) G(n, k)$ also returned by the WZ algorithm is immaterial here.  However, note that $p_0(k)$ and $p_1(k)$ are not zero for $0 \leq k < \spmoh$, so
$$\frac{-p_1(k)}{p_0(k)} = \frac{-4(4k+3)}{(4k+1) (2k+1)^2}.$$

\noindent
Use \eqref{e-t1-qk} to define
$$q(k)
= \frac{(-1)^k 4^k 4^k \ph{3/4}{k}}{4^k \ph{1/4}{k} 2^k 2^k \ph{1/2}{k}^2}
= \frac{(-1)^k \ph{3/4}{k}}{\ph{1/4}{k} \ph{1/2}{k}^2},$$

\noindent
which is hypergeometric in $k$ and satisfies both \eqref{e-t1-q} and $q(0) = 1$, as in the proof of Theorem \ref{t-wz-pair-polynomial-elimination}.   Set $\Fbar (n, k) = q(k) F(n, k)$ so that
\begin{equation} \label{e-h2-fbar}
\Fbar (n, k) = \frac{(-1)^k \ph{3/4}{k}}{\ph{1/4}{k} \ph{1/2}{k}^2} \cdot \frac{\ph{1/2}{n}^2 \ph{1/2}{n+k}}{\ph{1}{n}^2 \ph{1}{n-k}},
\end{equation}

\noindent
whence $\Fbar (n, 0) = F(n, 0)$.  Also set $\Gbar (n, k) = -q(k) G(n, k) / p_0(k)$, so
$$\Gbar (n, k) = - \frac{(-1)^k \ph{3/4}{k}}{\ph{1/4}{k} \ph{1/2}{k}^2} \cdot \frac{1}{(4k+1)(2k+1)^2} \cdot G(n, k).$$

\noindent
Running the WZ algorithm on $\Fbar (n, k)$ returns the difference operator $\Delta = K - 1$ as expected by Theorem \ref{t-wz-pair-polynomial-elimination}, as well as $\Gbar (n, k)$ from which $G(n, k)$ can be recovered.  Thus, it is not necessary to know $G(n, k)$ to proceed.  In fact,
$$\Gbar (n, k) =
- \frac{(-1)^k \ph{3/4}{k}}{\ph{1/4}{k} \ph{1/2}{k}^2}
\cdot \frac{1}{(4k+1)(2k+1)^2}
\cdot \frac{8 n^2 (n-k) \ph{1/2}{n}^2 \ph{1/2}{n+k}}{\ph{1}{n}^2 \ph{1}{n-k}}.$$

Summing the WZ pair $(K-1) \Fbar (n, k) = (N-1) \Gbar (n, k)$ over $n$ telescopes the right hand side, and using $\Gbar (0, k) = 0$ gives
\begin{equation} \label{e-h2-wz-pair-sum}
\sum_{n=0}^{\spmoh} \Fbar (n, k+1) - \sum_{n=0}^{\spmoh} \Fbar (n, k) = \Gbar (\spmoh+1, k) - \Gbar (0, k) = \Gbar (\spmoh+1, k).
\end{equation}

Any net factors of $p$ in the denominator of $\Gbar (n, k)$ will obstruct analyzing $\Gbar (n, k)$ modulo powers of $p$.  To prevent that, restrict $k$ to $0 \leq k < \spmoh$.  Moreover, only some of the terms in
$$\Gbar (\spmoh+1, k) = \frac{(p+1)^2 (2k - p - 1) \ph{1/2}{\spmoh+1}^2 \ph{1/2}{\spmoh+1+k} \ph{3/4}{k} (-1)^k}{(4k+1) (2k+1)^2 \ph{1}{\spmoh+1}^2 \ph{1}{\spmoh+1-k} \ph{1/2}{k}^2 \ph{1/4}{k}}$$

\noindent
contribute factors of $p$.  It is not too hard to see that $(p+1)^2$, $2k - p - 1$, $(-1)^k$, $2k+1$, $\ph{1}{\spmoh+1}$, $\ph{1}{\spmoh+1-k}$, and $\ph{1/2}{k}$ do not have factors of $p$ when $0 \leq k < \spmoh$.  Thus write,
\begin{equation} \label{e-h2-3mod4}
\Gbar (\spmoh+1, k) = \alpha \cdot \frac{\ph{1/2}{\spmoh+1}^2 \ph{1/2}{\spmoh+1+k} \ph{3/4}{k}}{(4k+1) \ph{1/4}{k}},
\end{equation}

\noindent
for $\alpha \in \Z_p^\times$.  Let $p' \in \Z$.  Of the factors in \eqref{e-h2-3mod4}, note $\ph{1/2}{\spmoh + 1}$ and $\ph{1/2}{\spmoh + 1 + k}$ each contribute $p$, the factor $\ph{3/4}{k}$ contributes $p$ only when $p = 4p'+3$ and $k > p'$, the factor $4k+1$ contributes $p$ only when $p = 4p'+1$ and $k = p'$, and $\ph{1/4}{k}$ contributes $p$ only when $p = 4p'+1$ and $k > p'$.

Consider the case $p = 4p' + 3$.  The restriction $0 \leq k < \spmoh$ prevents $(2k+1)^2$ from contributing a factor of $p^2$ in the denominator of $\Gbar (n, k)$, so $\Gbar (\spmoh+1, k)$ contains the factor $p^3$ for all such $k$ (that is, the contribution of $\ph{3/4}{k}$ when $k > p'$ is immaterial).  It follows that
$$\Gbar (\spmoh+1, k) \equiv 0 \pmod{p^3},$$

\noindent
and hence \eqref{e-h2-wz-pair-sum} implies
\begin{equation} \label{e-h2-f-chain}
\sum_{n=0}^{\spmoh} \Fbar (n, k+1) \equiv \sum_{n=0}^{\spmoh} \Fbar (n, k) \pmod{p^3}
\end{equation}

\noindent
for each $0 \leq k < \spmoh$.  In particular, using \eqref{e-h2-f-chain} for successive values of $k$ yields
\begin{equation} \label{e-h2-sum-f-chain}
\sum_{n=0}^{\spmoh} F(n)
= \sum_{n=0}^{\spmoh} \Fbar (n, 0)
\equiv \sum_{n=0}^{\spmoh} \Fbar (n, \spmoh) \pmod{p^3}.
\end{equation}

\noindent
The self nulling term $\ph{1}{n-k}$ in the denominator of $\Fbar (n, k)$ gives by \eqref{e-self-nulling} and \eqref{e-h2-fbar} that
$$\sum_{n=0}^{\spmoh} \Fbar (n, \spmoh)
= \Fbar (\spmoh, \spmoh)
= \frac{- \ph{1/2}{2 \spmoh} \ph{3/4}{\spmoh}}{\ph{1}{\spmoh}^2 \ph{1/4}{\spmoh}},$$

\noindent
so by \eqref{e-h2-sum-f-chain},
\begin{equation} \label{e-h2-key-reduction}
\sum_{n=0}^{\spmoh} F(n) \equiv \frac{- \ph{1/2}{2 \spmoh} \ph{3/4}{\spmoh}}{\ph{1}{\spmoh}^2 \ph{1/4}{\spmoh}} \pmod{p^3}.
\end{equation}

\noindent
Use Lemma \ref{l-ph-gp-identities-H} to rewrite $\Fbar (\spmoh, \spmoh)$ in terms of $\Gamma_p$ to obtain
\begin{equation} \label{e-h2-fbar-3mod4}
\Fbar (\spmoh, \spmoh)
= - \frac{-p}{2p-1} \cdot \frac{\Gamma_p \left ( \frac{1}{2} + p \right ) }{\Gamma_p \left ( \frac{1}{2} \right ) } \cdot \frac{p}{4} \cdot \frac{\Gamma_p \left ( \frac{1}{4} + \frac{p}{2} \right ) }{\Gamma_p \left ( \frac{3}{4} \right ) } \cdot \frac{1}{\Gamma_p \left ( \frac{1}{2} + \frac{p}{2} \right )^2} \cdot \frac{\Gamma_p \left ( \frac{1}{4} \right ) }{\Gamma_p \left ( \frac{-1}{4} + \frac{p}{2} \right ) }.
\end{equation}

\noindent
As $(2p-1)/4 \perp p$, by \eqref{e-gamma-p-props-2} it follows that
\begin{equation} \label{e-h2-gamma-p-transformation}
\Gamma_p \left ( \frac{-1}{4} + \frac{p}{2} \right ) (-1) \cdot \frac{2p-1}{4}
= \Gamma_p \left ( \frac{-1}{4} + \frac{p}{2} \right ) (-1) \left ( \frac{-1}{4} + \frac{p}{2} \right )
= \Gamma_p \left ( \frac{3}{4} + \frac{p}{2} \right ).
\end{equation}

\noindent
Moreover, using \eqref{e-gamma-p-reflection} and Lemma \ref{l-gamma-p-a0-reflection} gets
$$\Gamma_p(1/4) \Gamma_p(3/4) = (-1)^{a_0(3/4)} = (-1)^{a_0(1/4)},$$

\noindent
whence, in particular,
\begin{equation} \label{e-h2-gamma-p-squareds}
\frac{\Gamma_p(1/4)^2}{\Gamma_p(3/4)^2}
= \Gamma_p(1/4)^2 \cdot \frac{\Gamma_p(1/4)^2}{(-1)^{2 a_0(3/4)}} = \Gamma_p(1/4)^4.
\end{equation}

\noindent
Note that $\Gamma_p(1/2)^2 = (-1)^{\spmoh + 1} = 1$ by \eqref{e-gamma-p-props-3} as $\spmoh$ is odd.  Substitute \eqref{e-h2-gamma-p-transformation} and \eqref{e-h2-gamma-p-squareds} into \eqref{e-h2-fbar-3mod4}, then multiply and divide by suitable factors.  Rearranging yields
$$\Fbar (\spmoh, \spmoh)
= \frac{-p^2 \Gamma_p \left ( \frac{1}{4} \right )^4 }{16} \cdot \frac{\Gamma_p \left ( \frac{1}{2} + p \right ) }{\Gamma_p \left ( \frac{1}{2} \right ) } \cdot \frac{\Gamma_p \left ( \frac{1}{2} \right )^2}{\Gamma_p \left ( \frac{1}{2} + \frac{p}{2} \right )^2} \cdot \frac{\Gamma_p \left ( \frac{1}{4} + \frac{p}{2} \right ) }{\Gamma_p \left ( \frac{1}{4} \right ) } \cdot \frac{\Gamma_p \left ( \frac{3}{4} \right ) }{\Gamma_p \left ( \frac{3}{4} + \frac{p}{2}\right )}.$$

\noindent
Assume $p \geq 5$.  By Lemma \ref{l-h2-gp-quotient-p3}, the second and third factors above are $1$ modulo $p^3$.  By Theorem \ref{t-gk-theorems} with $t = 0$, the fourth and fifth factors above are $1$ modulo $p$.  Using \eqref{e-mod-trick}, it follows from \eqref{e-h2-sum-f-chain} and \eqref{e-h2-key-reduction} that
\begin{equation} \label{e-h2-3-mod-4}
\sum_{n=0}^{\spmoh} F(n)
\equiv \Fbar  (\spmoh, \spmoh)
\equiv \frac{-p^2 \Gamma_p \left ( \frac{1}{4} \right )^4 }{16} \pmod{p^3}.
\end{equation}

\noindent
A quick calculation shows $\Gamma_3(1/4)^4 = 1$, and it follows that \eqref{e-h2-3-mod-4} holds for $p = 3$ as well.  This establishes and extends (H.2) modulo $p^3$ for $p \equiv 3 \pmod{4}$.

When $p = 4p'+1$, the factors of $\Gbar (\spmoh+1, k)$ which could contain factors of $p$ are
$$\frac{\ph{1/2}{\spmoh+1}^2 \ph{1/2}{\spmoh+1+k}}{(4k+1) \ph{1/4}{k}}.$$

\noindent
By the preceding analysis, the numerator has a factor of $p^3$.  The denominator has one factor of $p$ when $k = p'$ due to $4k+1$, and one factor of $p$ when $k > p'$ due to $\ph{1/4}{k}$.  So $\Gbar (\spmoh+1, k)$ has a factor of $p^2$ in all cases, and thus $\Gbar (\spmoh+1, k) \equiv 0 \pmod{p^2}$.  It remains to show that $\Fbar (\spmoh, \spmoh) \equiv - \Gamma_p(1/4)^4 \pmod{p^2}$.  In this case $\spmoh$ is even, hence
$$\Fbar (\spmoh, \spmoh) = \frac{\ph{1/2}{2 \spmoh} \ph{3/4}{\spmoh}}{\ph{1}{\spmoh}^2 \ph{1/4}{\spmoh}}.$$

\noindent
Use Lemma \ref{l-ph-gp-identities-H} to rewrite $\Fbar (\spmoh, \spmoh)$ in terms of $\Gamma_p$, yielding
\begin{equation} \label{e-h2-fbar-1mod4}
\Fbar (\spmoh, \spmoh)
= \frac{-4}{2p-1} \cdot \frac{\Gamma_p \left ( \frac{1}{2} + p \right ) \Gamma_p \left ( \frac{1}{4} \right ) \Gamma_p \left ( \frac{1}{4} + \frac{p}{2} \right )}{\Gamma_p \left ( \frac{1}{2} \right ) \Gamma_p \left ( \frac{1}{2} + \frac{p}{2} \right )^2 \Gamma_p \left ( \frac{-1}{4} + \frac{p}{2} \right ) \Gamma_p \left ( \frac{3}{4} \right )}.
\end{equation}

\noindent
Substitute \eqref{e-h2-gamma-p-transformation} into \eqref{e-h2-fbar-1mod4} and rearrange to reach
$$\Fbar (\spmoh, \spmoh)
= \frac{\Gamma_p \left ( \frac{1}{4} + \frac{p}{2} \right )}{\Gamma_p \left ( \frac{3}{4} \right )} \cdot \frac{\Gamma_p \left ( \frac{1}{4} \right )}{\Gamma_p \left ( \frac{3}{4} + \frac{p}{2} \right )} \cdot \frac{\Gamma_p \left (\frac{1}{2} + p \right )}{\Gamma_p \left ( \frac{1}{2} \right )} \cdot \frac{1}{\Gamma_p \left ( \frac{1}{2} + \frac{p}{2} \right )^2}.
$$

\noindent
As $p = 4p' + 1$, then $\Gamma_p(1/2)^2 = (-1)^{\frac{p+1}{2}} = -1$ by \eqref{e-gamma-p-props-3}.  Using this fact to add $\Gamma_p(1/2)^2$ above, multiplying and dividing by suitable factors gives
\begin{equation} \label{e-h2-1mod4}
\Fbar (\spmoh, \spmoh)
= - \frac{\Gamma_p \left ( \frac{1}{4} + \frac{p}{2} \right )}{\Gamma_p \left ( \frac{1}{4} \right )} \cdot \frac{\Gamma_p \left (\frac{3}{4} \right )}{\Gamma_p \left ( \frac{3}{4} + \frac{p}{2} \right )} \cdot \frac{\Gamma_p \left (\frac{1}{2} + p \right )}{\Gamma_p \left ( \frac{1}{2} \right )} \cdot \frac{\Gamma_p \left ( \frac{1}{2} \right )^2}{\Gamma_p \left ( \frac{1}{2} + \frac{p}{2} \right )^2} \cdot \frac{\Gamma_p \left ( \frac{1}{4} \right )^2}{\Gamma_p \left ( \frac{3}{4} \right )^2}.
\end{equation}

\noindent
The rightmost factor in \eqref{e-h2-1mod4} is $\Gamma_p(1/4)^4$ because of \eqref{e-h2-gamma-p-squareds}.  Now, since $p = 4p' + 1$ then $p \geq 5$, so by Theorem \ref{t-gk-theorems} and \eqref{e-gk-easy-identities},
\begin{eqnarray*}
\frac{\Gamma_p \left ( \frac{1}{4} + \frac{p}{2} \right )}{\Gamma_p \left ( \frac{1}{4} \right )} & \equiv & 1 + \frac{G_1(1/4)p}{2} \pmod{p^2}, \\
\frac{\Gamma_p \left ( \frac{3}{4} \right )}{\Gamma_p \left ( \frac{3}{4} + \frac{p}{2} \right )} & \equiv & \frac{1}{1 + \frac{G_1(1/4)p}{2}} \pmod{p^2}.
\end{eqnarray*}

\noindent
Hence, the product of the leftmost two factors in \eqref{e-h2-1mod4} is $1$ modulo $p^2$.  The remaining two factors in \eqref{e-h2-1mod4} are $1$ modulo $p^2$ because of Lemma \ref{l-h2-gp-quotient-p3}.  Altogether, using \eqref{e-h2-sum-f-chain} and \eqref{e-h2-key-reduction} gets
$$\sum_{n=0}^{\spmoh} F(n)
\equiv \Fbar (\spmoh, \spmoh)
\equiv - \Gamma_p(1/4)^4 \pmod{p^2}$$

\noindent
when $p \equiv 1 \pmod{4}$.  This completes the proof.
\end{proof}

\section{Applications of Theorem \ref{t-wz-pair-polynomial-elimination}} \label{s-applications}

This section uses Theorem \ref{t-wz-pair-polynomial-elimination} to prove (C.2), (E.2), (F.2), (G.2), and (B.2).  The proofs start similarly to that of (H.2).  Given a summand as in \eqref{e-van-hamme},
$$F(n) = u(n) c^n \cdot \frac{\ph{1/a}{n}^m}{\ph{1}{n}^m},$$

\noindent
generalizing \eqref{e-wz-device-h2} define
\begin{equation} \label{e-wz-device-pm}
F(n, k) = u(n) c^n \cdot \frac{\ph{1/a}{n}^{m-1} \ph{1/a}{n+k}}{\ph{1}{n}^{m-1} \ph{1}{n-k}},
\end{equation}

\noindent
which satisfies $F(n, 0) = F(n)$.  In these cases, running the WZ algorithm on $F(n, k)$ produces a $G(n, k)$ and a difference operator $\Delta$ of the form \eqref{e-Delta-shape} with $d = \deg \Delta > 1$.  Thus, Theorem \ref{t-wz-pair-polynomial-elimination} does not directly apply to the resulting generalized WZ pair $F(n, k)$ and $G(n, k)$.  However, motivated by the proof of Theorem \ref{t-wz-pair-polynomial-elimination}, one can modify $F(n, k)$ such that the WZ algorithm will return a linear difference operator for which Theorem \ref{t-wz-pair-polynomial-elimination} does apply, while still being able to recover $F(n)$.  This phenomenon will be referred to as \emph{degree collapse}.

In particular, in each of these cases the difference operator $\Delta$ of the form \eqref{e-Delta-shape} has the property that $p_0(k)$ contains the factor $(ak + 1)^{m-1}$.  So, in the spirit of \eqref{e-thm-transform}, modify $F(n, k)$ by setting
\begin{equation} \label{e-deg-transform}
\Fdeg (n, k) = \frac{F(n, k)}{\ph{1/a}{k}^{m-1}},
\end{equation}

\noindent
which satisfies $\Fdeg (n, 0) = F(n)$.  Running the WZ algorithm on $\Fdeg (n, k)$ yields a new difference operator $\Delta$ of degree $1$, so the (H.2) proof method can still apply as shown in the proofs below.

\begin{proof}[Proof of Theorem \ref{t-van-hamme}, (C.2)]
When $p = 3$, a quick computation shows that
$$\sum_{n=0}^1 (4n+1) \cdot \frac{\ph{1/2}{n}^4}{\ph{1}{n}^4} = \frac{21}{16} \equiv 3 \pmod{27}.$$

\noindent
So, let $p \geq 5$ and using \eqref{e-wz-device-pm} set
$$F(n, k) = \frac{(4n+1) \ph{1/2}{n}^3 \ph{1/2}{n+k}}{\ph{1}{n}^3 \ph{1}{n-k}}.$$

\noindent
Running the WZ algorithm on $F(n, k)$ returns a quartic difference operator $\Delta$ of the form \eqref{e-Delta-shape} with polynomial coefficients
\begin{eqnarray*}
p_4(k) & = & 3600, \\
p_3(k) & = & 9776 k^2 + 41944 k + 50448, \\
p_2(k) & = & 7728 k^4 + 47952 k^3 + 110932 k^2 + 109628 k + 36204, \\
p_1(k) & = & 2 (33 k^3 - 285 k^2 - 1045 k - 863) (2k + 3)^3, \\
p_0(k) & = & -(4k + 5) (4k + 3) (2k + 3)^3 (2k + 1)^3.
\end{eqnarray*}

\noindent
Theorem \ref{t-wz-pair-polynomial-elimination} does not apply as $\deg \Delta > 1$, so use \eqref{e-deg-transform} to define
$$\Fdeg (n, k) = \frac{(4n+1) \ph{1/2}{n}^3 \ph{1/2}{n+k}}{\ph{1}{n}^3 \ph{1}{n-k} \ph{1/2}{k}^3}.$$

\noindent
Running the WZ algorithm on $\Fdeg(n, k)$ yields the difference operator
$$\Delta = -(k + 1) K - 1.$$

\noindent
Use \eqref{e-t1-qk} to obtain $q(k) = (-1)^k \ph{1}{k}$, and following the proof of Theorem \ref{t-wz-pair-polynomial-elimination} define
$$\Fbar (n, k) = \frac{(-1)^k (4n+1) \ph{1/2}{n}^3 \ph{1/2}{n+k} \ph{1}{k}}{\ph{1}{n}^3 \ph{1}{n-k} \ph{1/2}{k}^3}.$$

\noindent
As expected, running the WZ algorithm on $\Fbar (n, k)$ returns $\Delta = K-1$ with
$$\Gbar (n, k) = - \frac{16 n^3 (n-k) (-1)^k \ph{1/2}{n}^3 \ph{1/2}{n+k} \ph{1}{k}}{(2k+1)^3 \ph{1}{n}^3 \ph{1}{n-k} \ph{1/2}{k}^3}.$$

\noindent
Thus,
\begin{equation} \label{e-fbar-wz-pair}
\Fbar (n, k+1) - \Fbar (n, k) = \Gbar (n+1, k) - \Gbar (n, k),
\end{equation}

\noindent
and since $\Gbar (0, k) = 0$,
\begin{equation} \label{e-fbar-wz-pair-sum-c2}
\sum_{n=0}^{\spmoh} \Fbar (n, k+1) - \sum_{n=0}^{\spmoh} \Fbar (n, k) = \Gbar (\spmoh + 1, k).
\end{equation}

\noindent
Evaluating $\Gbar (\spmoh+1, k)$ results in
\begin{equation} \label{e-c2-gbar}
\Gbar (\spmoh+1, k) = - \frac{(p+1)^3 (p + 1 - 2k) (-1)^k \ph{1/2}{\spmoh+1}^3 \ph{1/2}{\spmoh+1+k} \ph{1}{k}}{(2k+1)^3 \ph{1}{\spmoh+1}^3 \ph{1}{\spmoh+1-k} \ph{1/2}{k}^3}.
\end{equation}

\noindent
Restricting $k$ to $0 \leq k < \spmoh$ prevents factors of $p$ in the denominator of \eqref{e-c2-gbar}.  In the numerator of \eqref{e-c2-gbar}, the only factors of $p$ come from $\ph{1/2}{\spmoh+1}^3 \ph{1/2}{\spmoh+1+k}$ which contributes a factor of $p^4$.  It follows that $\Gbar (\spmoh+1, k) \equiv 0 \pmod{p^4}$, and so \eqref{e-fbar-wz-pair-sum-c2} yields that
\begin{equation} \label{e-c2-chain}
\sum_{n=0}^{\spmoh} \Fbar (n, k+1) \equiv \sum_{n=0}^{\spmoh} \Fbar (n, k) \pmod{p^4}
\end{equation}

\noindent
for $0 \leq k < \spmoh$.  Setting $k = \spmoh-1$, the sum $\sum_{n=0}^{\spmoh} \Fbar (n, \spmoh)$ reduces to $\Fbar (\spmoh, \spmoh)$ by \eqref{e-self-nulling} due to the self nulling factor $\ph{1}{n-k}$ in $\Fbar (n, k)$'s denominator.  Hence, it remains to show that
$$\Fbar (\spmoh, \spmoh) = \frac{(-1)^{\spmoh} (2p-1) \ph{1/2}{2 \spmoh}}{\ph{1}{\spmoh}^2} \equiv p \pmod{p^3},$$

\noindent
because \eqref{e-c2-chain} implies
$$\sum_{n=0}^{\spmoh} F(n) \equiv \sum_{n=0}^{\spmoh} \Fbar (n, \spmoh) \equiv \Fbar (\spmoh, \spmoh) \pmod{p^4}.$$

Use Lemma \ref{l-ph-gp-identities-H} to rewrite $\Fbar (\spmoh, \spmoh)$ in terms of $\Gamma_p$ rather than Pochhammer symbols, then multiply and divide by $\Gamma_p(1/2)^2$ to obtain
$$\Fbar (\spmoh, \spmoh)
= \frac{(-1)^{\fpmoh}}{\Gamma_p(1/2)^2} \cdot (2p-1) \cdot \frac{p}{2} \cdot \frac{\Gamma_p(-1/2 + p)}{\Gamma_p(1/2)} \cdot \frac{\Gamma_p(1/2)^2}{\Gamma_p \left ( \frac{1}{2} + \frac{p}{2} \right )^2}.$$

\noindent
Simplify the leftmost $(-1)^{\fpmoh} / \Gamma_p(1/2)^2 = (-1)^{\fpmoh - \fppoh}$ to $-1$ with \eqref{e-gamma-p-props-3} and rearrange to get
$$\Fbar (\spmoh, \spmoh) = p \cdot \frac{-(2p-1)}{2} \cdot \frac{\Gamma_p(-1/2 + p)}{\Gamma_p(1/2)} \cdot \frac{\Gamma_p(1/2)^2}{\Gamma_p \left ( \frac{1}{2} + \frac{p}{2} \right )^2}.$$

\noindent
As $2p-1 \perp p$, use \eqref{e-gamma-p-props-2} to absorb $-(2p-1)/2$ into $\Gamma_p(-1/2 + p)$ to obtain $\Gamma_p(1/2 + p)$ and thus
$$
\Fbar (\spmoh, \spmoh)
= p \cdot \frac{\Gamma_p(1/2 + p)}{\Gamma_p(1/2)} \cdot \frac{\Gamma_p(1/2)^2}{\Gamma_p \left ( \frac{1}{2} + \frac{p}{2} \right )^2}.
$$

\noindent
Finally, use Lemma \ref{l-h2-gp-quotient-p3} to conclude that
$$\Fbar (\spmoh, \spmoh)
= p \cdot \frac{\Gamma_p(1/2 + p)}{\Gamma_p(1/2)} \cdot \frac{\Gamma_p(1/2)^2}{\Gamma_p \left ( \frac{1}{2} + \frac{p}{2} \right )^2}
\equiv p \pmod{p^3}.$$

\noindent
This completes the proof.
\end{proof}

\begin{proof}[Proof of Theorem \ref{t-van-hamme}, (E.2)]
The statement provides $p \equiv 1 \pmod{6}$, hence $p > 5$ and $\spmot$ is even.  Using \eqref{e-wz-device-pm}, set
$$F(n, k) = (6n+1) (-1)^n \cdot \frac{\ph{1/3}{n}^2 \ph{1/3}{n+k}}{\ph{1}{n}^2 \ph{1}{n-k}}.$$

\noindent
Running the WZ algorithm on $F(n, k)$ returns a cubic difference operator $\Delta$ of the form \eqref{e-Delta-shape} with polynomial coefficients
\begin{eqnarray*}
p_3(k) & = & 10584, \\
p_2(k) & = & 22626 k^2 + 77121 k + 70794, \\
p_1(k) & = & 6 (250 k^2 + 563 k + 340) (3k + 4)^2, \\
p_0(k) & = & (6k + 7) (3k + 4)^2 (3k + 2) (3k + 1)^2.
\end{eqnarray*}

\noindent
Theorem \ref{t-wz-pair-polynomial-elimination} does not apply as $\deg \Delta > 1$, so use \eqref{e-deg-transform} to define
$$\Fdeg (n, k) = (6n+1) (-1)^n \cdot \frac{\ph{1/3}{n}^2 \ph{1/3}{n+k}}{\ph{1}{n}^2 \ph{1}{n-k} \ph{1/3}{k}^2}.$$

\noindent
Running the WZ algorithm on $\Fdeg(n, k)$ yields the difference operator $\Delta = K + 1$.  Use \eqref{e-t1-qk} to obtain $q(k) = (-1)^k$, and following the proof of Theorem \ref{t-wz-pair-polynomial-elimination} define
$$\Fbar (n, k) = (6n+1) (-1)^{n+k} \cdot \frac{\ph{1/3}{n}^2 \ph{1/3}{n+k}}{\ph{1}{n}^2 \ph{1}{n-k} \ph{1/3}{k}^2},$$

\noindent
whence the WZ algorithm returns $\Delta = K - 1$ and
$$\Gbar (n, k) = - \frac{27 n^2 (n-k) (-1)^{n+k} \ph{1/3}{n}^2 \ph{1/3}{n+k}}{(3k+1)^2 \ph{1}{n}^2 \ph{1}{n-k} \ph{1/3}{k}^2}.$$

\noindent
It follows that \eqref{e-fbar-wz-pair} holds, and since $\Gbar (0, k) = 0$,
\begin{equation} \label{e-fbar-wz-pair-sum-e2}
\sum_{n=0}^{\spmot} \Fbar (n, k+1) - \sum_{n=0}^{\spmot} \Fbar (n, k) = \Gbar (\spmot + 1, k).
\end{equation}

\noindent
Restricting the values taken by $k$ to $0 \leq k < \spmot$ prevents factors of $p$ in the denominator of $\Gbar (n, k)$.  Looking at
$$\Gbar (\spmot+1, k) = - \frac{(p+2)^2 (3k-p-2) (-1)^{\spmot+1+k} \ph{1/3}{\spmot+1}^2 \ph{1/3}{\spmot+1+k}}{(3k+1)^2 \ph{1}{\spmot+1}^2 \ph{1}{\spmot+1-k} \ph{1/3}{k}^2},$$

\noindent
observe there are no factors of $p$ in the denominator, and that the numerator has a factor of $p^3$ due to $\ph{1/3}{\spmot+1}^2 \ph{1/3}{\spmot+1+k}$.  It follows that $\Gbar (\spmot+1, k) \equiv 0 \pmod{p^3}$, and so \eqref{e-fbar-wz-pair-sum-e2} yields that
$$\sum_{n=0}^{\spmot} \Fbar (n, k+1) \equiv \sum_{n=0}^{\spmot} \Fbar (n, k) \pmod{p^3}$$

\noindent
for $0 \leq k < \spmot$.  Due to the self nulling factor $\ph{1}{n-k}$ in $\Fbar (n, k)$'s denominator, substituting $k = \spmot-1$ and using \eqref{e-self-nulling} gives
$$\sum_{n=0}^{\spmot} \Fbar (n, \spmot) = \Fbar (\spmot, \spmot) = \frac{(2p-1) \ph{1/3}{2 \spmot}}{\ph{1}{\spmot}^2}.$$

\noindent
It suffices to show that $\Fbar (\spmot, \spmot) \equiv p \pmod{p^3}$.  The term $\ph{1/3}{2 \spmot}$ has the factor $p/3$, so rewrite $\Fbar (\spmot, \spmot)$ using \eqref{e-gamma-p-props-2} and that $\Gamma_p(1)^2 = 1$ to get
$$\Fbar (\spmot, \spmot)
= p \cdot \frac{2p-1}{3} \cdot \frac{\Gamma_p(1/3 + 2 \spmot)}{\Gamma_p(1/3)} \cdot \frac{1}{\Gamma_p(\spmot+1)^2}.
$$

\noindent
Since $(2p - 1)/3 \perp p$, absorb this factor into $\Gamma_p(1/3 + 2 \spmot)$ with \eqref{e-gamma-p-props-2} to reach
$$\Fbar (\spmot, \spmot)
= -p \cdot \frac{\Gamma_p(1/3 + 2 \spmot + 1)}{\Gamma_p(1/3)} \cdot \frac{1}{\Gamma_p(\spmot+1)^2}.
$$

\noindent
Rearrange using \eqref{e-gamma-p-reflection} and Lemma \ref{l-gamma-p-a0-reflection} in anticipation of applying Theorem \ref{t-gk-theorems}, obtaining
$$\Fbar (\spmot, \spmot)
= -p (-1)^{a_0(2/3)} \cdot \frac{\Gamma_p \left ( \frac{2}{3} + \frac{2p}{3} \right )}{\Gamma_p \left ( \frac{2}{3} \right )} \cdot \frac{\Gamma_p \left ( \frac{2}{3} \right )^2}{\Gamma_p \left ( \frac{2}{3} + \frac{p}{3} \right )^2}.$$

\noindent
Note that \eqref{e-a0-2o3-1mod6} implies $(-1)^{a_0(2/3)} = -1$, and so
\begin{equation} \label{e-e2-f-factors}
\Fbar (\spmot, \spmot)
= p \cdot \frac{\Gamma_p \left ( \frac{2}{3} + \frac{2p}{3} \right )}{\Gamma_p \left ( \frac{2}{3} \right )} \cdot \frac{\Gamma_p \left ( \frac{2}{3} \right )^2}{\Gamma_p \left ( \frac{2}{3} + \frac{p}{3} \right )^2}.
\end{equation}

\noindent
Apply Theorem \ref{t-gk-theorems} with $t=1$ to obtain
\begin{equation} \label{e-e2-mod-trick}
\frac{\Gamma_p \left ( \frac{2}{3} + \frac{2p}{3} \right )}{\Gamma_p \left ( \frac{2}{3} \right )} \cdot \frac{\Gamma_p \left ( \frac{2}{3} \right )^2}{\Gamma_p \left ( \frac{2}{3} + \frac{p}{3} \right )^2}
\equiv \frac{1 + G_1(2/3) \cdot \frac{2p}{3}}{\left ( 1 + G_1(2/3) \cdot \frac{p}{3} \right )^2}
\equiv 1 \pmod{p^2}.
\end{equation}

\noindent
After substituting \eqref{e-e2-mod-trick} into \eqref{e-e2-f-factors}, it follows from \eqref{e-mod-trick} that $\Fbar (\spmot, \spmot) \equiv p \pmod{p^3}$, which completes the proof.
\end{proof}

\begin{proof}[Proof of Theorem \ref{t-van-hamme}, (F.2)]
Note that $p \equiv 1 \pmod{4}$ implies $p \geq 5$.  Using \eqref{e-wz-device-pm}, set
$$F(n, k) = (8n+1) (-1)^n \cdot \frac{\ph{1/4}{n}^2 \ph{1/4}{n+k}}{\ph{1}{n}^2 \ph{1}{n-k}}.$$

\noindent
Running the WZ algorithm on $F(n, k)$ returns a cubic difference operator $\Delta$ of the form \eqref{e-Delta-shape} with polynomial coefficients
\begin{eqnarray*}
p_3(k) & = & 270400, \\
p_2(k) & = & 557184 k^2 + 1827232 k + 1627760, \\
p_1(k) & = & 4 (4737 k^2 + 9986 k + 5465) (4k + 5)^2, \\
p_0(k) & = & (8k + 9) (8k + 5) (4k + 5)^2  (4k + 1)^2.
\end{eqnarray*}

\noindent
Theorem \ref{t-wz-pair-polynomial-elimination} does not apply as $\deg \Delta > 1$, so use \eqref{e-deg-transform} to define
$$\Fdeg (n, k) = (8n+1) (-1)^n \cdot \frac{\ph{1/4}{n}^2 \ph{1/4}{n+k}}{\ph{1}{n}^2 \ph{1}{n-k} \ph{1/4}{k}^2}.$$

\noindent
Running the WZ algorithm on $\Fdeg(n, k)$ yields the difference operator $\Delta = K + 1$.  Use \eqref{e-t1-qk} to obtain $q(k) = (-1)^k$, and following the proof of Theorem \ref{t-wz-pair-polynomial-elimination} define
$$\Fbar (n, k) = (8n+1) (-1)^{n+k} \cdot \frac{\ph{1/4}{n}^2 \ph{1/4}{n+k}}{\ph{1}{n}^2 \ph{1}{n-k} \ph{1/4}{k}^2},$$

\noindent
whence the WZ algorithm returns $\Delta = K - 1$ and
$$\Gbar (n, k) = \frac{64 n^2 (n-k) (-1)^{n+k} \ph{1/4}{n}^2 \ph{1/4}{n+k}}{(4k+1)^2 \ph{1}{n}^2 \ph{1}{n-k} \ph{1/4}{k}^2}.$$

\noindent
It follows that \eqref{e-fbar-wz-pair} holds, and since $\Gbar (0, k) = 0$,
\begin{equation} \label{e-fbar-wz-pair-sum-f2}
\sum_{n=0}^{\spmoq} \Fbar (n, k+1) - \sum_{n=0}^{\spmoq} \Fbar (n, k) = \Gbar (\spmoq + 1, k).
\end{equation}

\noindent
Restricting the values taken by $k$ to $0 \leq k < \spmoq$ prevents unwanted factors of $p$ in the denominator of $\Gbar (n, k)$.  Examining
$$\Gbar (\spmoq+1, k) = \frac{(p+3)^2 (p + 3 - 4k) (-1)^{\spmoq+1+k} \ph{1/4}{\spmoq+1}^2 \ph{1/4}{\spmoq+1+k}}{(4k+1)^2 \ph{1}{\spmoq+1}^2 \ph{1}{\spmoq+1-k} \ph{1/4}{k}^2}$$

\noindent
shows that $\Gbar (\spmoq+1, k) \equiv 0 \pmod{p^3}$ since the denominator has no factors of $p$ and the numerator has a factor of $p^3$ due to $\ph{1/4}{\spmoq+1}^2 \ph{1/4}{\spmoq+1+k}$.  Consequently, by \eqref{e-fbar-wz-pair-sum-f2},
$$\sum_{n=0}^{\spmoq} \Fbar (n, k+1) \equiv \sum_{n=0}^{\spmoq} \Fbar (n, k) \pmod{p^3}.$$

\noindent
Due to the self nulling factor $\ph{1}{n-k}$ in $\Fbar (n, k)$'s denominator, substituting $k = \spmoq - 1$ and using \eqref{e-self-nulling} yields
$$\sum_{n=0}^{\spmoq} \Fbar (n, \spmoq) = \Fbar (\spmoq, \spmoq).$$

\noindent
Hence, the proof reduces to showing that
$$\Fbar (\spmoq, \spmoq)
= \frac{(2p-1) \ph{1/4}{2 \spmoq}}{\ph{1}{\spmoq}^2}
\equiv \frac{-p}{\Gamma_p(1/4) \Gamma_p(3/4)} \pmod{p^3}.$$

\noindent
Rearrange using \eqref{e-gamma-p-props-2} in preparation to use Theorem \ref{t-gk-theorems} as follows.  Note that $\ph{1/4}{2 \spmoq}$ has a factor of $p/4$, and no other factors of $p$.  Absorb $1/4 + 2 \spmoq = (2p-1)/4$ into $\Gamma_p(1/4 + 2 \spmoq)$ at the cost of a minus sign.  Recall $\Gamma_p(1)^2 = 1$.  Multiply and divide by $\Gamma_p(3/4)^2$.  In other words,
\begin{equation} \label{e-f2-fbar}
\Fbar (\spmoq, \spmoq)
= \frac{-p}{\Gamma_p \left ( \frac{1}{4} \right ) \Gamma_p \left ( \frac{3}{4} \right )} \cdot \frac{\Gamma_p \left ( \frac{3}{4} + \frac{p}{2} \right )}{\Gamma_p \left ( \frac{3}{4} \right )} \cdot \frac{\Gamma_p \left ( \frac{3}{4} \right )^2}{\Gamma_p \left ( \frac{3}{4} + \frac{p}{4} \right )^2}.
\end{equation}

\noindent
Using Theorem \ref{t-gk-theorems} with $t=1$ gives
\begin{equation} \label{e-f2-fbar-factors}
\frac{\Gamma_p \left ( \frac{3}{4} + \frac{p}{2} \right )}{\Gamma_p \left ( \frac{3}{4} \right )} \cdot \frac{\Gamma_p \left ( \frac{3}{4} \right )^2}{\Gamma_p \left ( \frac{3}{4} + \frac{p}{4} \right )^2}
\equiv \frac{1 + G_1(3/4) \cdot \frac{p}{2} }{ \left ( 1 + G_1(3/4) \cdot \frac{p}{4} \right )^2}
\equiv 1 \pmod{p^2}.
\end{equation}

\noindent
Substituting \eqref{e-f2-fbar-factors} into \eqref{e-f2-fbar} then using \eqref{e-mod-trick} completes the proof.
\end{proof}

\begin{proof}[Proof of Theorem \ref{t-van-hamme}, (G.2)]
The following extends (G.2) modulo $p^4$ as shown in Swisher \cite{Swisher15}.  Note that $p \equiv 1 \pmod{4}$ implies $p \geq 5$.  Using \eqref{e-wz-device-pm}, set
$$F(n, k) = (8n+1) \cdot \frac{\ph{1/4}{n}^3 \ph{1/4}{n+k}}{\ph{1}{n}^3 \ph{1}{n-k}}.$$

\noindent
Running the WZ algorithm on $F(n, k)$ returns a quartic difference operator $\Delta$ of the form \eqref{e-Delta-shape} with polynomial coefficients
\begin{eqnarray*}
p_4(k) & = & 16646400, \\
p_3(k) & = & 49677056 k^2 + 210799808 k + 256053120, \\
p_2(k) & = & 49152768 k^4 + 318509952 k^3 + 812974576 k^2 + 962132816 k + 443357040, \\
p_1(k) & = & 4 (61953 k^3 + 182019 k^2 + 174371 k + 53153) (4 k + 5)^3, \\
p_0(k) & = & -(8 k + 9) (8 k + 5) (4 k + 5)^3  (4 k + 1)^3.
\end{eqnarray*}

\noindent
Theorem \ref{t-wz-pair-polynomial-elimination} does not apply as $\deg \Delta > 1$, so use \eqref{e-deg-transform} to define
$$\Fdeg (n, k) = (8n+1) \cdot \frac{\ph{1/4}{n}^3 \ph{1/4}{n+k}}{\ph{1}{n}^3 \ph{1}{n-k} \ph{1/4}{k}^3}.$$

\noindent
Running the WZ algorithm on $\Fdeg(n, k)$ yields the difference operator
$$\Delta = (2k+1)K + 2.$$

\noindent
Use \eqref{e-t1-qk} to obtain $q(k) = (-1)^k \ph{1/2}{k}$, and following the proof of Theorem \ref{t-wz-pair-polynomial-elimination} define
$$\Fbar (n, k) = (8n+1) (-1)^k \cdot \frac{\ph{1/4}{n}^3 \ph{1/4}{n+k} \ph{1/2}{k}}{\ph{1}{n}^3 \ph{1}{n-k} \ph{1/4}{k}^3},$$

\noindent
whence the WZ algorithm returns $\Delta = K - 1$ and
$$\Gbar (n, k) = - \frac{256 n^3 (n-k) (-1)^k \ph{1/4}{n}^3 \ph{1/4}{n+k} \ph{1/2}{k}}{(4k+1)^3 \ph{1}{n}^3 \ph{1}{n-k} \ph{1/4}{k}^3}.$$

\noindent
It follows that \eqref{e-fbar-wz-pair} holds, and since $\Gbar (0, k) = 0$,
\begin{equation} \label{e-fbar-wz-pair-sum-g2}
\sum_{n=0}^{\spmoq} \Fbar (n, k+1) - \sum_{n=0}^{\spmoq} \Fbar (n, k) = \Gbar (\spmoq + 1, k).
\end{equation}

\noindent
Restricting the values taken by $k$ to $0 \leq k < \spmoq$ prevents unwanted factors of $p$ in the denominator of $\Gbar (n, k)$.  Examining
$$\Gbar (\spmoq+1, k) = - \frac{(p+3)^3 (p + 3 - 4k) (-1)^k \ph{1/4}{\spmoq+1}^3 \ph{1/4}{\spmoq+1+k} \ph{1/2}{k}}{(4k+1)^3 \ph{1}{\spmoq+1}^3 \ph{1}{\spmoq+1-k} \ph{1/4}{k}^3}$$

\noindent
shows that $\Gbar (\spmoq+1, k) \equiv 0 \pmod{p^4}$ since the denominator has no factors of $p$ and the numerator has a factor of $p^4$ due to $\ph{1/4}{\spmoq+1}^3 \ph{1/4}{\spmoq+1+k}$.  Consequently, by \eqref{e-fbar-wz-pair-sum-g2},
$$\sum_{n=0}^{\spmoq} \Fbar (n, k+1) \equiv \sum_{n=0}^{\spmoq} \Fbar (n, k) \pmod{p^4}.$$

\noindent
Due to the self nulling factor $\ph{1}{n-k}$ in $\Fbar (n, k)$'s denominator, substituting $k = \spmoq - 1$ and using \eqref{e-self-nulling} yields
$$\sum_{n=0}^{\spmoq} \Fbar (n, \spmoq) = \Fbar (\spmoq, \spmoq).$$

\noindent
Hence, the proof reduces to showing that
\begin{equation} \label{e-g2-proof-reduction}
\Fbar (\spmoq, \spmoq)
= \frac{(2p-1) (-1)^{\spmoq} \ph{1/4}{2 \spmoq} \ph{1/2}{\spmoq}}{\ph{1}{\spmoq}^3}
\equiv \frac{p \Gamma_p(1/2) \Gamma_p(1/4)}{\Gamma_p(3/4)} \pmod{p^4}.
\end{equation}

\noindent
Before continuing, for reasons that will become apparent shortly, set $p = 4p'+1$ for an appropriate integer $p'$, whence $\spmoq = p'$.  Perform this replacement only for $(-1)^{\spmoq}$ in \eqref{e-g2-proof-reduction}, such that
$$\Fbar (\spmoq, \spmoq) = \frac{(2p-1) (-1)^{p'} \ph{1/4}{2 \spmoq} \ph{1/2}{\spmoq}}{\ph{1}{\spmoq}^3}.$$

\noindent
Now use \eqref{e-gamma-p-props-2} to rearrange in preparation to use Theorem \ref{t-gk-theorems} as follows.  Note that $\ph{1/4}{2 \spmoq}$ has no factor of $p$ except $p/4$.  Absorb $1/4 + 2 \spmoq = (2p-1)/4$ into $\Gamma_p(1/4 + 2 \spmoq)$ at the cost of a minus sign.  Recall $\Gamma_p(1)^3 = -1$, and that also $\Gamma_p(1/2)^2 = (-1)^{(p+1)/2} = -1$ by \eqref{e-gamma-p-props-3} as here $p \equiv 1 \pmod{4}$.  Multiply and divide by $\Gamma_p(3/4)^3$.  In other words,
$$\Fbar (\spmoq, \spmoq)
= p (-1)^{p'+1} \cdot \frac{\Gamma_p \left ( \frac{3}{4} + \frac{p}{2} \right )}{\Gamma_p \left ( \frac{1}{4} \right ) } \cdot \frac{\Gamma_p\left ( \frac{1}{4} + \frac{p}{4} \right )}{\Gamma_p \left ( \frac{1}{2} \right )} \cdot \frac{\Gamma_p \left ( \frac{1}{2} \right )^2}{\Gamma_p \left ( \frac{3}{4} + \frac{p}{4} \right )^3} \cdot \frac{\Gamma_p \left ( \frac{3}{4} \right )^3}{\Gamma_p \left ( \frac{3}{4} \right )^3},$$

\noindent
whence rearranging gives
\begin{equation} \label{e-g2-factors}
\Fbar (\spmoq, \spmoq) = \frac{p (-1)^{p'+1} \Gamma_p(1/2)}{\Gamma_p(3/4)^2} \cdot \frac{\Gamma_p \left ( \frac{3}{4} + \frac{p}{2} \right )}{\Gamma_p \left ( \frac{3}{4} \right )} \cdot \frac{\Gamma_p \left ( \frac{1}{4} + \frac{p}{4} \right )}{\Gamma_p \left ( \frac{1}{4} \right )} \cdot \frac{\Gamma_p \left( \frac{3}{4} \right )^3}{\Gamma_p \left ( \frac{3}{4} + \frac{p}{4} \right )^3}.
\end{equation}

\noindent
The first term in \eqref{e-g2-factors} can be transformed to match the right hand side of \eqref{e-g2-proof-reduction}.  By \eqref{e-gamma-p-reflection} one has that $\Gamma_p(1/4) \Gamma_p(3/4) = (-1)^{a_0(3/4)}$.  Noting that $a_0(3/4) = p'+1$ as $4(p'+1) \equiv 4p'+4 \equiv 3 \pmod{p}$, $$\frac{p (-1)^{p'+1} \Gamma_p(1/2)}{\Gamma_p(3/4)^2}
= \frac{p (-1)^{2(p'+1)} \Gamma_p(1/2) \Gamma_p(1/4)}{\Gamma_p(3/4)}
= \frac{p \Gamma_p(1/2) \Gamma_p(1/4)}{\Gamma_p(3/4)}.$$

\noindent
Hence, by \eqref{e-mod-trick} it suffices to show the rightmost three terms in \eqref{e-g2-factors} are congruent to $1$ modulo $p^3$.  Using Theorem \ref{t-gk-theorems} and \eqref{e-gk-easy-identities}, calculation yields
\begin{eqnarray*}
& & \frac{\Gamma_p \left ( \frac{3}{4} + \frac{p}{2} \right )}{\Gamma_p \left ( \frac{3}{4} \right )} \cdot \frac{\Gamma_p \left ( \frac{1}{4} + \frac{p}{4} \right )}{\Gamma_p \left ( \frac{1}{4} \right )} \cdot \frac{\Gamma_p \left ( \frac{3}{4} \right )^3}{\Gamma_p \left ( \frac{3}{4} + \frac{p}{4} \right )^3} \\
& \equiv & \frac{ \left ( 1 + G_1(3/4) \cdot \frac{p}{2} + G_2(3/4) \cdot \frac{p^2}{8} \right ) \left ( 1 + G_1(1/4) \cdot \frac{p}{4} + G_2(1/4) \cdot \frac{p^2}{32} \right )}{\left ( 1 + G_1(3/4) \cdot \frac{p}{4} + G_2(3/4) \cdot \frac{p^2}{32} \right)^3} \\
& \equiv & \frac{1 + G_1(3/4) \cdot \frac{3p}{4} + G_2(1/4) \cdot \frac{3p^2}{32} + G_2(3/4) \cdot \frac{3p^2}{16}}{1 + G_1(3/4) \cdot \frac{3p}{4} + G_2(1/4) \cdot \frac{3p^2}{32} + G_2(3/4) \cdot \frac{3p^2}{16}} \pmod{p^3},
\end{eqnarray*}

\noindent
which is $1$ modulo $p^3$ as desired.  This completes the proof.
\end{proof}

So far, these are WZ method proofs of supercongruences originally solved by other methods.  This approach can also streamline existing WZ method proofs, such as Zudilin's original proof of (B.2) \cite{Zudilin09}.

\begin{proof}[Proof of Theorem \ref{t-van-hamme}, (B.2)]
A quick calculation shows the congruence holds for $p = 3$, so let $p \geq 5$.  Using \eqref{e-wz-device-pm}, set
\begin{equation} \label{e-b2-generalize}
F(n, k) = \frac{(4n+1)(-1)^n \ph{1/2}{n}^2 \ph{1/2}{n+k}}{\ph{1}{n}^2 \ph{1}{n-k}}.
\end{equation}

\noindent
Running the WZ algorithm on $F(n, k)$ returns a cubic difference operator $\Delta$ of the form \eqref{e-Delta-shape} with polynomial coefficients
\begin{eqnarray*}
p_3(k) & = & 648, \\
p_2(k) & = & 1552 k^2 + 5816 k + 5796, \\
p_1(k) & = & 2 (145 k^2 + 386 k + 285) (2k + 3)^2 , \\
p_0(k) & = & (4k + 5) (4k + 3) (2k + 3)^2  (2k + 1)^2.
\end{eqnarray*}

\noindent
Theorem \ref{t-wz-pair-polynomial-elimination} does not apply as $\deg \Delta > 1$, so use \eqref{e-deg-transform} to define
$$\Fdeg (n, k) = \frac{(4n+1)(-1)^n \ph{1/2}{n}^2 \ph{1/2}{n+k}}{\ph{1}{n}^2 \ph{1}{n-k} \ph{1/2}{k}^2}.$$

\noindent
Running the WZ algorithm on $\Fdeg(n, k)$ yields the difference operator $\Delta = K + 1$.  Use \eqref{e-t1-qk} to obtain $q(k) = (-1)^k$, and following the proof of Theorem \ref{t-wz-pair-polynomial-elimination} define
$$\Fbar (n,k) = \frac{(4n+1)(-1)^{n+k} \ph{1/2}{n}^2 \ph{1/2}{n+k}}{\ph{1}{n}^2 \ph{1}{n-k} \ph{1/2}{k}^2},$$

\noindent
whence the WZ algorithm provides $\Delta = K - 1$ and
$$\Gbar (n, k) = \frac{8n^2 (n-k) (-1)^{n+k} \ph{1/2}{n}^2 \ph{1/2}{n+k}}{(2k+1)^2 \ph{1}{n}^2 \ph{1}{n-k} \ph{1/2}{k}^2}.$$

\noindent
It follows that \eqref{e-fbar-wz-pair} holds, and since $\Gbar (0, k) = 0$,
\begin{equation} \label{e-fbar-wz-pair-sum-b2}
\sum_{n=0}^{\spmoh} \Fbar (n, k+1) - \sum_{n=0}^{\spmoh} \Fbar (n, k) = \Gbar (\spmoh + 1, k).
\end{equation}

\noindent
Restricting the values taken by $k$ to $0 \leq k < \spmoh$ prevents unwanted factors of $p$ in the denominator of $\Gbar (n, k)$.  Examining
$$\Gbar (\spmoh+1, k) = \frac{(p+1)^2 (p+ 1 - 2k) (-1)^{\spmoh+1+k} \ph{1/2}{\spmoh+1}^2 \ph{1/2}{\spmoh+1+k}}{(2k+1)^2 \ph{1}{\spmoh+1}^2 \ph{1}{\spmoh+1-k} \ph{1/2}{k}^2}.$$

\noindent
shows that $\Gbar (\spmoh+1, k) \equiv 0 \pmod{p^3}$ since the denominator has no factors of $p$ and the numerator has a factor of $p^3$ due to $\ph{1/2}{\spmoh+1}^2 \ph{1/2}{\spmoh+1+k}$.  Consequently, by \eqref{e-fbar-wz-pair-sum-b2},
$$\sum_{n=0}^{\spmoh} \Fbar (n, k+1) \equiv \sum_{n=0}^{\spmoh} \Fbar (n, k) \pmod{p^3}.$$

\noindent
Due to the self nulling factor $\ph{1}{n-k}$ in $\Fbar (n, k)$'s denominator, substituting $k = \spmoh - 1$ and using \eqref{e-self-nulling} yields
$$\sum_{n=0}^{\spmoh} \Fbar (n, \spmoh) = \Fbar (\spmoh, \spmoh).$$

\noindent
So, it suffices to calculate $\Fbar (\spmoh, \spmoh)$ modulo $p^3$.  By Lemma \ref{l-ph-gp-identities-H}, \eqref{e-gamma-p-props-2}, and Lemma \ref{l-h2-gp-quotient-p3},
$$\Fbar (\spmoh, \spmoh)
= \frac{-p}{\Gamma_p(1/2)^2} \cdot \frac{\Gamma_p(1/2 + p)}{\Gamma_p(1/2)} \cdot \frac{\Gamma_p(1/2)^2}{\Gamma_p\left ( \frac{1}{2} + \frac{p}{2} \right )^2}
\equiv \frac{-p}{\Gamma_p(1/2)^2} \pmod{p^3}.$$

\noindent
This completes the proof.
\end{proof}

\section{WZ devices} \label{s-wz-devices}

Suppose the WZ algorithm returns a difference operator $\Delta$ of the form \eqref{e-Delta-shape} with $\deg \Delta > 1$ for a certain $F(n, k)$.  The factor required to change $F(n, k)$ toward degree collapse as in Section \ref{s-applications} may not divide $p_0(k)$.  Sometimes these factors can be solved for without too much difficulty.  Consider the proof of (B.2) given in Section \ref{s-applications}.  Running the WZ algorithm on $F(n, k)$ from \eqref{e-b2-generalize} returns a cubic difference operator $\Delta$.  Hoping to find suitable factors, construct the $\Fdeg (n, k)$ ansatz
$$\Fdeg (n,k) = \frac{(4n+1)(-1)^{n+k} \ph{1/2}{n}^2 \ph{1/2}{n+k}}{\ph{1}{n}^2 \ph{1}{n-k} \ph{a/b}{k} \ph{c/d}{k}}.$$

\noindent
for suitable integer variables $a, b, c, d$.  Running the WZ algorithm on $\Fdeg (n, k)$ responds with
$$\Delta = 4(bk+a)(dk+c) K - bd (2k+1)^2,$$

\noindent
from which \eqref{e-t1-qk} in the proof of Theorem \ref{t-wz-pair-polynomial-elimination} says to define
$$\Fbar(n, k)
= \Fdeg (n, k) q(k)
= \Fdeg (n, k) \cdot \frac{\ph{a/b}{k} \ph{c/d}{k}}{\ph{1/2}{k}^2}
= F(n, k) \cdot \frac{1}{\ph{1/2}{k}^2}.$$

\noindent
Running the WZ algorithm on $\Fbar(n, k)$ returns $\Delta = K - 1$, and $\Fbar(n, 0) = F(n, 0)$.  Note the resulting $\Fbar(n, k)$ does not depend on $a, b, c, d$.  So, given an initial ansatz close enough to the shape required to trigger degree collapse, the function $q(k)$ may be able to substitute the vital factor for the unhelpful bits.  Hence, in a sense, there are few ways to change $F(n, k)$ so that the WZ algorithm returns a linear difference operator.  Of course, there are no guarantees that a tentative ansatz will trigger the intended degree collapse as the example just shown.

Likewise, suppose that for a certain $F(n, k)$, hypergeometric in both $n$ and $k$, the WZ algorithm returns a difference operator $\Delta$ of the form \eqref{e-Delta-shape} with $d = \deg \Delta > 1$ and hence Theorem \ref{t-wz-pair-polynomial-elimination} does not apply.  Assuming there is no division by zero, define
$$q(n, k) = \frac{p_1(k)}{p_0(k)} + \sum_{j=2}^d \frac{p_j(k)}{p_0(k)} \cdot \frac{F(n, k+j)}{F(n, k+1)},$$

\noindent
whence by \eqref{e-gen-wz-pair}
$$\frac{\Delta F(n, k)}{p_0(k)}
= F(n, k+1) q(n, k) + F(n, k)
= \frac{G(n+1, k)}{p_0(k)} - \frac{G(n, k)}{p_0(k)}.$$

\noindent
While $q(n, k)$ is rational in $k$, it likely depends on $n$.  So, setting $\Fdeg (n, k) = F(n, k) q(n, k) / q(n, 0)$ may not help the WZ algorithm return a linear difference operator $\Delta$, and recovering the original $F(n, k)$ summand could be difficult even if $\deg \Delta = 1$.  Notwithstanding, coupling these observations with the ideas behind Theorem \ref{t-wz-pair-polynomial-elimination} suggest the following definition.

\begin{definition}
Take $F(n)$, hypergeometric in $n$, and let $n$ take integer values in the possibly infinite interval $[s, t]$.  The function $w(n, k)$ is called a WZ device for $F(n)$ if it satisfies the following three conditions.
\begin{enumerate}
\item The function $w(n, k)$ is hypergeometric in $n$ and $k$.

\item The identity $w(n, 0) = 1$ holds for $n \in [s, t]$.

\item Running the WZ algorithm on $F(n) w(n, k)$ returns a linear difference operator $\Lambda$.
\end{enumerate}

\noindent
In addition, if $\mathbb{F} \subseteq \C$ is a field and the polynomial coefficients of $\Lambda$ split into linear factors in $\mathbb{F}[k]$, then the WZ device $w(n, k)$ is said to be splitting over $\mathbb{F}$.
\end{definition}

This definition encapsulates the creativity needed to find WZ pairs for summands such as those in \eqref{e-van-hamme}.  Let $F(n)$ be such a summand, and set $[t, s] = [0, \spmox{d}]$.  A WZ device of special interest here is a function $w(n, k)$, hypergeometric in both $n$ and $k$, satisfying $w(n, 0) = 1$ for $0 \leq n \leq \spmox{d}$, and such that the WZ algorithm returns $\Delta = K - 1$ when run on
$$\Fbar (n, k) = F(n) w(n, k).$$

\noindent
Note that, for a fixed $F(n)$, the WZ devices that split over $\mathbb{F}$ are essentially equivalent: they differ by a ratio of suitable $w(n, k)$ functions, assuming no division by zero.

The following proof of (D.2) exemplifies a WZ device pattern that does not arise from a starting application of \eqref{e-wz-device-pm}.

\begin{proof}[Proof of Theorem \ref{t-van-hamme}, (D.2)]
As $p \equiv 1 \pmod{6}$, assume $p \geq 5$.  Moreover, let $p = 6p'+1$ for an appropriate integer $p'$, and unlike in \eqref{e-wz-device-pm} define
$$F(n, k) = (6n+1) \cdot \frac{\ph{1/3}{n}^4 \ph{1/3}{n+k} \ph{1/3}{n-k}}{\ph{1}{n}^4 \ph{1}{n-k} \ph{1}{n+k}}.$$

\noindent
Running the WZ algorithm for $F(n, k)$ returns the difference operator
$$\Delta = (3k+2)^4 K - (3k+1)^4.$$

\noindent
Use \eqref{e-t1-qk} to obtain $q(k) = \ph{2/3}{k}^4 / \ph{1/3}{k}^4$, and following the proof of Theorem \ref{t-wz-pair-polynomial-elimination} define
$$\Fbar (n, k) = (6n+1) \cdot \frac{\ph{1/3}{n}^4 \ph{1/3}{n+k} \ph{1/3}{n-k} \ph{2/3}{k}^4}{\ph{1}{n}^4 \ph{1}{n-k} \ph{1}{n+k} \ph{1/3}{k}^4},$$

\noindent
whence the WZ algorithm returns $\Delta = K - 1$ and
$$\Gbar (n, k) = \frac{729 (2k+1) n^4 (n-k) \ph{1/3}{n}^4 \ph{1/3}{n+k} \ph{1/3}{n-k} \ph{2/3}{k}^4}{(3k+2-3n) (3k+1)^4 \ph{1}{n}^4 \ph{1}{n-k} \ph{1}{n+k} \ph{1/3}{k}^4}.$$

\noindent
It follows that \eqref{e-fbar-wz-pair} holds, and since $\Gbar (0, k) = 0$,
\begin{equation} \label{e-fbar-wz-pair-sum-d2}
\sum_{n=0}^{\spmot} \Fbar (n, k+1) - \sum_{n=0}^{\spmot} \Fbar (n, k) = \Gbar (\spmot + 1, k).
\end{equation}

\noindent
Restricting the values taken by $k$ to $0 \leq k < \spmot$ prevents factors of $p$ in the denominator of $\Gbar (n, k)$.  A quick analysis shows
$$\Gbar (\spmot+1, k) = \frac{3 (2k+1) (p+2)^4 (p+2-3k) \ph{1/3}{\spmot+1}^4 \ph{1/3}{\spmot+1+k} \ph{1/3}{\spmot+1-k} \ph{2/3}{k}^4}{(3k-p) (3k+1)^4 \ph{1}{\spmot+1}^4 \ph{1}{\spmot+1-k} \ph{1}{\spmot+1+k} \ph{1/3}{k}^4}$$

\noindent
has a net factor of $p^5$ in its numerator, and so $\Gbar (\spmot+1, k) \equiv 0 \pmod{p^5}$.  Consequently, by \eqref{e-fbar-wz-pair-sum-d2},
$$\sum_{n=0}^{\spmot} \Fbar (n, k+1) \equiv \sum_{n=0}^{\spmot} \Fbar (n, k) \pmod{p^5}.$$

\noindent
Due to the self nulling factor $\ph{1}{n-k}$ in $\Fbar (n, k)$'s denominator, substituting $k = \spmot-1$ and using \eqref{e-self-nulling} yields
$$\sum_{n=0}^{\spmot} \Fbar (n, \spmot) = \Fbar (\spmot, \spmot) = \frac{(2p-1) \ph{1/3}{2 \spmot} \ph{2/3}{\spmot}^4}{\ph{1}{\spmot}^4 \ph{1}{2 \spmot}}.$$

\noindent
It suffices to examine $\Fbar (\spmot, \spmot)$.  Operate as shown below using \eqref{e-gamma-p-props-2} to obtain
$$\Fbar (\spmot, \spmot)
= 3 \cdot \frac{(2p-1)}{3} \cdot \frac{\ph{1/3}{2 \spmot} \ph{2/3}{\spmot}^4}{\ph{1}{\spmot}^4 \ph{1}{2 \spmot}}
= 3 \cdot \frac{\ph{1/3}{2 \spmot + 1} \ph{2/3}{\spmot}^4}{\ph{1}{\spmot}^4 \ph{1}{2 \spmot}}.$$

\noindent
Rewrite $\Fbar (\spmot, \spmot)$ in terms of $\Gamma_p$.  Use that $2 \spmot + 1 = 4p'+1$ and $\Gamma_p(1) = -1$ to simplify the powers of $-1$ due to \eqref{e-gamma-p-props-2}.  Since $p = 6p' + 1$, converting $\ph{1/3}{2 \spmot + 1}$ with \eqref{e-gamma-p-props-2} yields an extra factor of $p/3$.  Finally, multiply and divide by appropriate evaluations of $\Gamma_p$ to later apply Theorem \ref{t-gk-theorems} and get
$$\Fbar (\spmot, \spmot)
= p \cdot \frac{\Gamma_p \left ( \frac{2}{3} + \frac{2p}{3} \right )}{\Gamma_p \left ( \frac{2}{3} \right )} \cdot \frac{\Gamma_p \left ( \frac{1}{3} + \frac{p}{3} \right )^4}{\Gamma_p \left ( \frac{1}{3} \right )^4} \cdot \frac{\Gamma_p \left ( \frac{2}{3} \right )^4}{\Gamma_p \left ( \frac{2}{3} + \frac{p}{3} \right )^4} \cdot \frac{\Gamma_p \left ( \frac{1}{3} \right )}{\Gamma_p \left ( \frac{1}{3} + \frac{2p}{3} \right )} \cdot \frac{\Gamma_p \left ( \frac{1}{3} \right )^2}{\Gamma_p \left ( \frac{2}{3} \right )^7}.$$

\noindent
Observe that $\Gamma_p(1/3) \Gamma_p(2/3) = (-1)^{a_0(2/3)} = -1$ by \eqref{e-gamma-p-reflection} and \eqref{e-a0-2o3-1mod6}.  Rearrangement yields
\begin{equation} \label{e-d2-factors}
\Fbar (\spmot, \spmot) = -p \Gamma_p(1/3)^9 \cdot \frac{\Gamma_p \left ( \frac{2}{3} + \frac{2p}{3} \right )}{\Gamma_p \left ( \frac{2}{3} \right )} \cdot \frac{\Gamma_p \left ( \frac{1}{3} + \frac{p}{3} \right )^4}{\Gamma_p \left ( \frac{1}{3} \right )^4} \cdot \frac{\Gamma_p \left ( \frac{2}{3} \right )^4}{\Gamma_p \left ( \frac{2}{3} + \frac{p}{3} \right )^4} \cdot \frac{\Gamma_p \left ( \frac{1}{3} \right )}{\Gamma_p \left ( \frac{1}{3} + \frac{2p}{3} \right )}.
\end{equation}

\noindent
Expand the rightmost four factors in \eqref{e-d2-factors} modulo $p^3$ via Theorem \ref{t-gk-theorems}, using \eqref{e-gk-binomial-4} as needed, to obtain
\begin{eqnarray*}
\frac{\Gamma_p \left ( \frac{2}{3} + \frac{2p}{3} \right )}{\Gamma_p \left ( \frac{2}{3} \right )} & \equiv & 1 + G_1(2/3) \cdot \frac{2p}{3} + G_2(2/3) \cdot \frac{2p^2}{9} \pmod{p^3}, \\
\frac{\Gamma_p \left ( \frac{1}{3} + \frac{p}{3} \right )^4}{\Gamma_p \left( \frac{1}{3} \right )^4} & \equiv & 1 + G_1(1/3) \cdot \frac{4p}{3} + G_1(1/3)^2 \cdot \frac{2p^2}{3} + G_2(1/3) \cdot \frac{2p^2}{9} \pmod{p^3}, \\
\frac{\Gamma_p \left ( \frac{2}{3} \right )^4}{\Gamma_p \left ( \frac{2}{3} + \frac{p}{3} \right )^4} & \equiv & \frac{1}{1 + G_1(2/3) \cdot \frac{4p}{3} + G_1(2/3)^2 \cdot \frac{2p^2}{3} + G_2(2/3) \cdot \frac{2p^2}{9}} \pmod{p^3}, \\
\frac{\Gamma_p \left ( \frac{1}{3} \right )}{\Gamma_p \left ( \frac{1}{3} + \frac{2p}{3} \right )} & \equiv & \frac{1}{1 + G_1(1/3) \cdot \frac{2p}{3} + G_2(1/3) \cdot \frac{2p^2}{9}} \pmod{p^3}.
\end{eqnarray*}

\noindent
Multiplying these factors together forms a quotient.  Compute this quotient modulo $p^3$, then apply \eqref{e-gk-easy-identities} to make the evaluations of $G_1$ match to find both numerator and denominator are
$$1 + G_1(1/3) \cdot 2p + G_1(1/3)^2 \cdot \frac{14p^2}{9} + G_2(1/3) \cdot \frac{2p^2}{9} + G_2(2/3) \cdot \frac{2p^2}{9},$$

\noindent
whence the fraction amounts to $1$ modulo $p^3$.  Substitute this congruence in \eqref{e-d2-factors}, then use \eqref{e-mod-trick} to deduce $\Fbar (\spmot, \spmot) \equiv -p \Gamma_p(1/3)^9 \pmod{p^4}$, completing the proof.
\end{proof}

In summary, the proofs in the present paper rely on the WZ devices listed in Table \ref{table-wz-devices}.

\begin{table}[ht]
\caption{List of WZ devices used in this paper.} \label{table-wz-devices}
\begin{tabular}{r|l}
Van Hamme supercongruence & WZ device \\
\hline
\hlineFix
(B.2) & $\frac{\ph{1}{n} \ph{1/2}{n+k}}{\ph{1/2}{n} \ph{1}{n-k}} \cdot \frac{(-1)^k}{\ph{1/2}{k}^2}$ \\
\hlineFix
(C.2) & $\frac{\ph{1}{n} \ph{1/2}{n+k}}{\ph{1/2}{n} \ph{1}{n-k}} \cdot \frac{(-1)^k \ph{1}{k}}{\ph{1/2}{k}^3}$ \\
\hlineFix
(E.2) & $\frac{\ph{1}{n} \ph{1/3}{n+k}}{\ph{1/3}{n} \ph{1}{n-k}} \cdot \frac{(-1)^k}{\ph{1/3}{k}^2}$ \\
\hlineFix
(F.2) & $\frac{\ph{1}{n} \ph{1/4}{n+k}}{\ph{1/4}{n} \ph{1}{n-k}} \cdot \frac{(-1)^k}{\ph{1/4}{k}^2}$ \\
\hlineFix
(G.2) & $\frac{\ph{1}{n} \ph{1/4}{n+k}}{\ph{1/4}{n} \ph{1}{n-k}} \cdot \frac{(-1)^k \ph{1/2}{k}}{\ph{1/4}{k}^3}$ \\
\hlineFix
(H.2) & $\frac{\ph{1}{n} \ph{1/2}{n+k}}{\ph{1/2}{n} \ph{1}{n-k}} \cdot \frac{(-1)^k \ph{3/4}{k}}{\ph{1/4}{k} \ph{1/2}{k}^2}$ \\
\hline
\hlineFix
(D.2) & $\frac{\ph{1}{n}^2 \ph{1/3}{n+k} \ph{1/3}{n-k}}{\ph{1/3}{n}^2 \ph{1}{n-k} \ph{1}{n+k}} \cdot \frac{\ph{2/3}{k}^4}{\ph{1/3}{k}^4}$
\end{tabular}
\end{table}

\raggedright
\bibliographystyle{plain}
\bibliography{project-bibliography}

\end{document}